\documentclass{article}   	% use "amsart" instead of "article" for AMSLaTeX format
\usepackage{geometry}               % See geometry.pdf to learn the layout options. There are lots.
\usepackage{graphicx,color}						
\usepackage[dvipsnames]{xcolor}
\usepackage{amssymb,amsmath,amsthm,amsfonts}
\usepackage{bm}                     % mathbf greek
\usepackage[english]{babel}
\usepackage[utf8]{inputenc}

\usepackage[pdfencoding=auto,colorlinks]{hyperref}
\hypersetup{linkcolor=blue,citecolor=blue,filecolor=black,urlcolor=blue}

\usepackage[showonlyrefs]{mathtools}
\mathtoolsset{showonlyrefs=true}
\usepackage{newtxtext,newtxmath,helvet,courier}
\usepackage{fix-cm}
\usepackage{tikz}
\usepackage{pgfplots}
\usepgflibrary{arrows.meta}
\usetikzlibrary{patterns}
\usepgfplotslibrary{colormaps}
\usetikzlibrary{pgfplots.colormaps}

\pgfplotsset{compat=1.11}

\tikzset{> = {Stealth[inset=0pt]}}
\newtheorem{proposition}{Proposition}[section]
\newtheorem{theorem}[proposition]{Theorem}

\newtheorem{lemma}[proposition]{Lemma}

\theoremstyle{definition}

\theoremstyle{remark}
\newtheorem{remark}[proposition]{Remark}
\numberwithin{equation}{section}
\newcommand{\eps}{\varepsilon}

\newcommand{\N}{{\mathbb{N}}}

\newcommand{\R}{{\mathbb{R}}}
\renewcommand{\C}{{\mathbb{C}}}

\newcommand{\icomp}{{\mathrm{i}}}

\newcommand{\Bcal}{{\mathcal{B}}}

\newcommand{\Ecal}{{\mathcal{E}}}

\newcommand{\Gcal}{{\mathcal{G}}}

\newcommand{\Mcal}{{\mathcal{M}}}

\newcommand{\Qcal}{{\mathcal{Q}}}

\newcommand{\Ucal}{{\mathcal{U}}}

\newcommand{\lamax}{\dd}
\newcommand{\dd}{\Lambda}

\usepackage[normalem]{ulem}
\usepackage{indentfirst}

\DeclareMathOperator{\dist}{dist}

\def\sideremark#1{\ifvmode\leavevmode\fi\vadjust{\vbox to0pt{\vss% the remark
 \hbox to 0pt{\hskip\hsize\hskip1em%                          will appear only
 \vbox{\hsize2.1cm\tiny\raggedright\pretolerance10000%          on the side
  \noindent #1\hfill}\hss}\vbox to15pt{\vfil}\vss}}}%

\title{Normalized solutions for Nonlinear Schr\"odinger systems on bounded domains}
\author{Benedetta Noris, Hugo Tavares, Gianmaria Verzini}
\begin{document}
\maketitle

\begin{abstract}
We analyze $L^2$-normalized solutions of nonlinear Schr\"odinger systems of Gross-Pitaevskii type, on bounded domains, with homogeneous Dirichlet boundary conditions. We provide sufficient conditions 
for the existence of orbitally stable standing waves. Such waves correspond to global minimizers of the associated energy in the $L^2$-subcritical and critical cases, and to local ones in the $L^2$-supercritical case. 
Notably, our study includes also the Sobolev-critical case.
\end{abstract}

\noindent
{\footnotesize \textbf{AMS-Subject Classification}}. %Primary:\\
35Q55, %NLS-like equations (nonlinear Schr\"odinger)\\
35B33, %Critical exponents\\
35B35, %Stability\\
%35C08 %Soliton solutions\\
35J50. %Variational methods for elliptic systems\\
%Secondary:\\
%35B09 Positive solutions\\
%35J57 Boundary value problems for second-order elliptic systems\\
{\footnotesize }

\section{Introduction}

In this paper, we carry on the study of normalized solutions for Nonlinear Schr\"odinger (NLS) equations and systems, started in \cite{ntvAnPDE,ntvDCDS}.

Let $\Omega\subset\R^N$, $N\geq 1$, be a bounded smooth domain, $\mu_1,\mu_2>0$ and $\beta\in\R$. We consider the following system of coupled Gross-Pitaevskii equations
\begin{equation}\label{eq:system_schro}
\begin{cases}
\icomp\partial_t\Psi_1 + \Delta\Psi_1 + \Psi_1( \mu_1 |\Psi_1|^{p-1} +\beta  |\Psi_1|^{(p-3)/2}|\Psi_2|^{(p+1)/2} )=0\\
\icomp\partial_t\Psi_2 + \Delta\Psi_2 + \Psi_2( \mu_2 |\Psi_2|^{p-1} +\beta  |\Psi_2|^{(p-3)/2}|\Psi_1|^{(p+1)/2} )=0
\end{cases}
\end{equation}
with $\Psi_i:\R^+\times\Omega\to\C$ and, for every $t>0$, $\Psi_i(t,\cdot)\in H^1_0(\Omega;\C)$ ($i=1,2$). Throughout the paper we will distinguish several cases in the range
\[
\begin{cases}
p>1 & N=1,2,\\
1<p\le 2^*-1 & N\ge 3,
\end{cases}
\]
where $2^*=2N/(N-2)$ denotes the Sobolev critical exponent.

NLS systems with power-type nonlinearities appear in several different physical models from quantum mechanics, in particular when $p = 3$ or $p = 5$.
Such models include Bose–Einstein condensation in multiple hyperfine spin states \cite{PhysRevLett.81.5718,MR2090357} and the propagation of mutually
incoherent waves packets in nonlinear optics \cite{agrawal2000}. Moreover, both the cases $\Omega = \R^N$ and
$\Omega$ bounded are of interest \cite{FibichMerle2001,Fukuizumi2012}, the latter one appearing also as a limiting case of the system on $\R^N$ with
(confining) trapping potential.

System \eqref{eq:system_schro} preserves, at least formally, both the masses
\[
\Qcal(\Psi_i)=\int_\Omega |\Psi_i|^2 \qquad i=1,2,
\]
and the energy
\[
\Ecal(\Psi_1,\Psi_2) := \frac12\int_{\Omega}|\nabla \Psi_1|^2+|\nabla \Psi_2|^2 -
\frac{1}{p+1}\int_\Omega \mu_1 |\Psi_1|^{p+1} + 2\beta |\Psi_1|^{(p+1)/2} |\Psi_2|^{(p+1)/2} + \mu_2 |\Psi_2|^{p+1}.
\]
We look for standing wave solutions $(\Psi_1(t,x),\Psi_2(t,x))=(e^{\icomp \omega_1 t}u_1(x),e^{\icomp \omega_2 t}u_2(x))$ of \eqref{eq:system_schro} such that $(u_1,u_2) \in H^1_0(\Omega;\R^2)$ and
\begin{equation}\label{eq:mass_constraint}
\Qcal(u_i)=\rho_i, \quad i=1,2,
\end{equation}
for some $\rho_1,\rho_2\ge0$ prescribed a priori. Then, $(u_1,u_2)$ is a normalized solution of an elliptic system; namely, there exist $(\omega_1,\omega_2)\in \R^2$ such that
\begin{equation}\label{eq:system_elliptic}
\begin{cases}
-\Delta u_1+ \omega_1 u_1=\mu_1 u_1|u_1|^{p-1}+\beta u_1|u_1|^{(p-3)/2} |u_2|^{(p+1)/2}\\
-\Delta u_2+ \omega_2 u_2=\mu_2 u_2|u_2|^{p-1}+\beta u_2 |u_2|^{(p-3)/2} |u_1|^{(p+1)/2}\\
\int_\Omega u_i^2=\rho_i, \quad i=1,2,\\
(u_1,u_2) \in H^1_0(\Omega;\R^2).
\end{cases}
\end{equation}
Solutions of \eqref{eq:system_elliptic} can be seen as critical points of $\Ecal$, constrained to the Hilbert manifold
\begin{equation}\label{eq:defM}
\Mcal =\Mcal_{\rho_1,\rho_2} := \left\{(u_1,u_2)\in H^1_0(\Omega;\R^2):\int_\Omega u_i^2=\rho_i, \ i=1,2 \right\},
\end{equation}
in which case the unknowns $\omega_i$ play the role of Lagrange multipliers. Our main aim is to
provide conditions on $p$ and $(\rho_1,\rho_2)$ (and also on $\mu_1,\mu_2,\beta$) so that
$\left.\Ecal\right|_{\Mcal}$ admits minima, either global or local. We call such solutions
\emph{least energy solutions}, or \emph{ground states}. Secondly, we consider the stability properties of such ground states, with respect
to the evolution system \eqref{eq:system_schro}.

An alternative, non equivalent point of view ---which we do not treat here--- is that of considering the parameters $\omega_i$ in \eqref{eq:system_elliptic} as fixed,
without any normalizing condition on the functions $u_i$.
This leads to an alternative definition of ground states, that of \emph{least action solutions}: for a detailed discussion of this topic we refer the interested reader to
the introduction of \cite{ntvAnPDE}. Starting from \cite{MR2135447,MR2358296,MR2263573,MR2252973,MR2302730,Sirakov2007,MR1939088,MR2629888}, the literature dealing with this approach is vast and we do not
even make an attempt to summarize it here. As a matter of fact, the results for non-normalized solutions cannot be directly extended to the normalized ones: among the other
reasons, because in the latter case the ambient space is the Hilbert manifold $\Mcal$ (rather than a vector space).

Going back to normalized solutions, the simplest case one can face is that of a single NLS equation on $\R^N$, with a pure power nonlinearity. In such case, the problem can be completely solved by
simple scaling arguments. This structure breaks down whenever one considers a system, as well as non-homogeneous nonlinearities, bounded domains or confining potentials. Apart when global minimization 
can be applied, see  \cite{RoseWeinstein88}, as far as we know the first
result in the literature is due to Jeanjean \cite{MR1430506}, for the superlinear, Sobolev-subcritical
NLS single equation on $\R^N$ with a non-homogeneous nonlinearity. In recent years, other papers appeared, dealing with the NLS equation or system, always in the Sobolev subcritical regime,  either on $\R^N$ \cite{MR3539467,MR3534090,MR3639521,MR3638314,MR3777573,0951-7715-31-5-2319,2017arXiv170302832B}
or on a bounded domain \cite{MR2928850,ntvAnPDE,ntvDCDS,MR3660463,MR3689156,BonheureJeanjeanNoris}.
These two settings are rather different in nature: each one requires a specific approach, and the results are in general not comparable.
A key difference is that $\R^N$ is invariant under translations and dilations, which has pros and cons: on the one hand, translations are responsible for a loss of compactness;
on the other hand, in the Sobolev subcritical case, dilations can be used to produce variations and eventually construct natural constraints such as the so-called Pohozaev manifold.
This tool is not available when working in bounded domains, and also the gain of compactness is lost when we face the Sobolev critical case.

However, a common key tool in the study of normalized solution is the Gagliardo-Nirenberg inequality (see \eqref{eq:GN} below), which can be used to estimate the non-quadratic part in $\Ecal$
in terms of the quadratic one. As a consequence, the exponent $p$ in \eqref{eq:system_elliptic} can be classified according to the following four cases:
\begin{itemize}
\item[(H1)] superlinear, $L^2$--subcritical: $1<p<1+4/N$;
\item[(H2)] $L^2$--critical: $p=1+{4/N}$;
\item[(H3)] $L^2$-supercritical, Sobolev--subcritical: $1+4/N<p<1 + 4/(N-2)^+=2^*-1$;
\item[(H4)] Sobolev--critical: $p=2^*-1$, for $N\geq 3$.
\end{itemize}
In the first three cases, the study of the single equation
\begin{equation}\label{eq:singleeq}
\begin{cases}
-\Delta u_1 + \omega_1 u=\mu_1 u_1|u_1|^{p-1}\\
\int_\Omega u_1^2=\rho_1,\quad u_1\in H^1_0(\Omega),
\end{cases}
\end{equation}
has been carried on in \cite{ntvAnPDE,MR3689156}. Notice that \eqref{eq:singleeq} is a particular case of \eqref{eq:system_elliptic}, when $\rho_2=0$, with associated energy
$u_1\mapsto \Ecal(u_1,0)$. Summarizing, it is known that
\begin{itemize}
 \item (H1) implies that \eqref{eq:singleeq} has a solution, which is a global minimizer, for $\rho_1\ge0$;
 \item (H2) implies that \eqref{eq:singleeq} has a solution, which is a global minimizer, for $0\le\rho_1<\rho_*(\Omega,N,p,\mu_1)<+\infty$;
 \item (H3) implies that \eqref{eq:singleeq} has a solution, which is a local minimizer, for $0\le\rho_1<\rho_*(\Omega,N,p,\mu_1)<+\infty$, and a second one of mountain pass type.
\end{itemize}
Moreover, all the minimizers above are associated to orbitally stable solitary waves of the corresponding evolutive equation.

Up to our knowledge, the only paper dealing with the NLS system \eqref{eq:system_elliptic} (with both $\rho_i>0$) is  \cite{ntvDCDS}. Among other things, in that paper we deal with the
$L^2$-supercritical, Sobolev--subcritical case (H3), obtaining the existence of orbitally stable solitary waves, in case both $\rho_1$, $\rho_2$ are sufficiently small and $\rho_1/\rho_2$ is uniformly bounded away from $0$ and $+\infty$. This result is perturbative in nature, the existence following by a multi-parametric extension of a Ambrosetti-Prodi-type reduction \cite{AmbrosettiProdiBook} and the stability by the
Grillakis-Shatah-Strauss stability theory \cite{GrillakisShatahStrauss}.
The corresponding solutions are close to suitably normalized
first eigenfunctions of the Dirichlet laplacian.

The aim of the present paper is twofold: on the one hand, in the cases (H1)-(H2)-(H3), we extend to systems the above described results obtained in
\cite{ntvAnPDE,MR3689156} for the single equation; on the other hand, we treat for the first time the Sobolev critical case (H4), obtaining results which are
new also in the case of a single equation. Now we describe in details our results.\medskip

In what follows, we take
\[
\Omega\subset \R^N \text{ a Lipschitz bounded  domain ($N\geq 1$),}\qquad \mu_1,\mu_2>0,\qquad \beta\in \R,
\]
and denote by
$C_N$ the best constant appearing in the Gagliardo-Nirenberg inequality in the $L^2$-critical case (see \eqref{eq:def_CN} ahead) while $S_N$ is the best constant appearing in the Sobolev inequality (see \eqref{eq:Sobolev_constant}).

To start with, as we already mentioned, both (H1) and (H2) can be treated in a quite standard way by using the Gagliardo-Nirenberg inequality.
Even though this result is somewhat expected, we provide it here since we could not find a precise reference to cite.
\begin{theorem}[$L^2$--subcritical and $L^2$--critical cases: existence and stability]\label{prop:subcritical}
 Suppose that one of the following cases occurs
\begin{itemize}
\item[(i)] $1<p<1+{4/N}$ and $\rho_1,\rho_2\ge0$;
\item[(ii)] $p=1+{4/N}$ and $\rho_1,\rho_2\ge0$ are such that (see Fig. \ref{fig:N=2})
\begin{equation}\label{N=2condition2}
\max\left\{\mu_1 \rho_1^{2/N},\ \mu_2 \rho_2^{2/N},\ \mu_1\rho_1^{2/N}+\mu_2 \rho_2^{2/N} +\frac{NC_N}{N+2}     \left((\beta^+)^2-\mu_1 \mu_2\right)(\rho_1\rho_2)^{2/N}\right\}<\frac{N+2}{NC_{N}}
\end{equation}

\end{itemize}
Then:
\begin{itemize}
\item[a)] the level $\inf_{\Mcal}\Ecal$ is achieved at some $(u_1,u_2)\in \Mcal$, which is a non-negative solution of \eqref{eq:system_elliptic} for some $(\omega_1,\omega_2)\in \R^2$ ($u_i>0$ when $\rho_i>0$);
\item[b)]  the set of ground states
\[
G=\left\{ (u_1,u_2) \in H^1_0(\Omega;\C^2): \ (|u_1|,|u_2|)\in\Mcal,\ \Ecal(u_1,u_2)=
\inf_{\Mcal}\Ecal \right\}
\]
is conditionally orbitally stable.
\end{itemize}
\end{theorem}
\begin{remark}
We recall the definition of orbital stability in Section \ref{sec:stab}. Actually, notice that we prove \emph{conditional} orbital stability,
where the condition is that the solution of system \eqref{eq:system_schro}, with Cauchy datum $(\psi_{1},\psi_{2})\in H^1_0(\Omega;\C^2)$,
exists locally in time for a time interval which is uniform in $\|(\psi_{1},\psi_{2})\|_{H^1_0}$, and that $\Qcal$ and $\Ecal$ are preserved
along the solutions. This holds true under further assumptions on $\mu_i$, $\beta$ and $\Omega$, see also \cite{ntvDCDS} and references therein;
however, being the field so vast, even a rough summary of well-posedness for Schr\"odinger systems on bounded domains is far beyond the scopes
of this paper.
\end{remark}
\begin{remark}\label{rem:uniform_beta1}
Observe that, for $\beta\leq 0$,  \eqref{N=2condition2} reduces to
\begin{equation*}
\max\left\{\mu_1 \rho_1^{2/N},\ \mu_2 \rho_2^{2/N}\right\}<\frac{N+2}{NC_{N}}.
\end{equation*}
which is independent from $\beta$.
\end{remark}
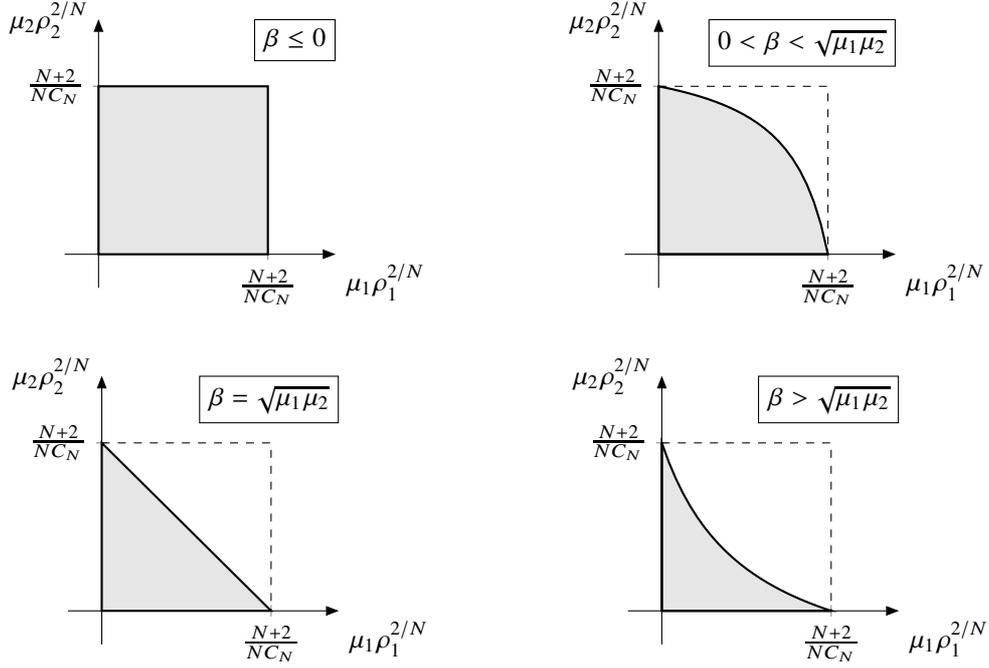
\begin{figure}[htbp]
\begin{center}
\begin{tikzpicture} \begin{axis}[
   axis lines=center,
   width=.35\linewidth,
   height=.35\linewidth,
   axis line style={->},
   typeset ticklabels with strut,
   x label style={below right},
   y label style={left},
   xmin=-.2, xmax=1.4,
   ymin=-.2, ymax=1.4,
   xlabel={$\mu_1\rho_1^{2/N}$},
   ylabel={$\mu_2\rho_2^{2/N}$},
   xtick={1},
   xticklabels={$\frac{N+2}{NC_N}$},
   ytick={1},
   yticklabels={$\frac{N+2}{NC_N}$},
   %\accuratedued{100},
   ]
\draw (1.4,1.4) node[below left, draw=black, fill=white] {$\beta \le0$};
\filldraw[thick,fill=gray!20!white] (0,0)--(0,1)--(1,1)--(1,0)--cycle;
\end{axis}
\end{tikzpicture}
\hspace{.1\linewidth}
\begin{tikzpicture} \begin{axis}[
   axis lines=center,
   width=.35\linewidth,
   height=.35\linewidth,
   axis line style={->},
   typeset ticklabels with strut,
   x label style={below right},
   y label style={left},
   xmin=-.2, xmax=1.4,
   ymin=-.2, ymax=1.4,
   xlabel={$\mu_1\rho_1^{2/N}$},
   ylabel={$\mu_2\rho_2^{2/N}$},
   xtick={1},
   xticklabels={$\frac{N+2}{NC_N}$},
   ytick={1},
   yticklabels={$\frac{N+2}{NC_N}$},
   %\accuratedued{100},
   ]
\draw (1.4,1.4) node[below left, draw=black, fill=white] {$0<\beta <\sqrt{\mu_1\mu_2}$};
\addplot[domain=0:1, thick, fill=gray!20!white]
 {(1-x)/(1-.8*x)}\closedcycle;
\draw[thick] (0,1)--(0,0)--(1,0);
\draw[dashed] (0,1)--(1,1)--(1,0);
\end{axis}
\end{tikzpicture}
\bigskip\\
\begin{tikzpicture} \begin{axis}[
   axis lines=center,
   width=.35\linewidth,
   height=.35\linewidth,
   axis line style={->},
   typeset ticklabels with strut,
   x label style={below right},
   y label style={left},
   xmin=-.2, xmax=1.4,
   ymin=-.2, ymax=1.4,
   xlabel={$\mu_1\rho_1^{2/N}$},
   ylabel={$\mu_2\rho_2^{2/N}$},
   xtick={1},
   xticklabels={$\frac{N+2}{NC_N}$},
   ytick={1},
   yticklabels={$\frac{N+2}{NC_N}$},
   %\accuratedued{100},
   ]
\draw (1.4,1.4) node[below left, draw=black, fill=white] {$\beta =\sqrt{\mu_1\mu_2}$};
\filldraw[thick,fill=gray!20!white] (0,0)--(0,1)--(1,0)--cycle;
\draw[dashed] (0,1)--(1,1)--(1,0);
\end{axis}
\end{tikzpicture}
\hspace{.1\linewidth}
\begin{tikzpicture} \begin{axis}[
   axis lines=center,
   width=.35\linewidth,
   height=.35\linewidth,
   axis line style={->},
   typeset ticklabels with strut,
   x label style={below right},
   y label style={left},
   xmin=-.2, xmax=1.4,
   ymin=-.2, ymax=1.4,
   xlabel={$\mu_1\rho_1^{2/N}$},
   ylabel={$\mu_2\rho_2^{2/N}$},
   xtick={1},
   xticklabels={$\frac{N+2}{NC_N}$},
   ytick={1},
   yticklabels={$\frac{N+2}{NC_N}$},
   %\accuratedued{100},
   ]
\draw (1.4,1.4) node[below left, draw=black, fill=white] {$\beta >\sqrt{\mu_1\mu_2}$};
\addplot[domain=0:1, thick, fill=gray!20!white]
 {(1-x)/(1+2*x)}\closedcycle;
\draw[thick] (0,1)--(0,0)--(1,0);
\draw[dashed] (0,1)--(1,1)--(1,0);
\end{axis}
\end{tikzpicture}
\caption{condition \eqref{N=2condition2} read in terms of $(\mu_1\rho_1^{2/N},\mu_2\rho_2^{2/N})$ (see also Remark \ref{rem:N=2} ahead).}\label{fig:N=2}
\end{center}
\end{figure}

Now we turn to the case in which either (H3) or (H4) hold true, i.e. when $1+4/N < p \le 2^*-1$ (no upper bound if $N=1,2$). Contrarily to the previous cases, in this one it is known that $\left.\Ecal\right|_{\Mcal}$ is not
bounded below, see Lemma \ref{lemma:geometry_mp} ahead. Nonetheless we will show that, even though
no global minima can exist, local ones do, in case $(\rho_1,\rho_2)$ belongs to some explicit set. To detect the existence of such minima we need to introduce some auxiliary problem and further notations. Let, for $\alpha\geq \lambda_1(\Omega)$ (the first Dirichlet eigenvalue of $-\Delta$ in $\Omega$),
\begin{equation}\label{eq:Balpha}
\begin{split}
\Bcal_\alpha &:= \left\{(u_1,u_2)\in\Mcal:\int_\Omega |\nabla u_1|^2+|\nabla u_2|^2 \leq
(\rho_1+\rho_2)\alpha \right\},\\
\Ucal_\alpha &:= \left\{(u_1,u_2)\in\Mcal:\int_\Omega |\nabla u_1|^2+
|\nabla u_2|^2=
(\rho_1+\rho_2)\alpha \right\}
\end{split}
\end{equation}
Notice that $\Bcal_\alpha$ is not empty, since it contains at least a pair of suitably normalized first eigenfunctions, and that $\Ucal_\alpha$ is the topological boundary of $\mathcal{B}_\alpha$ in $\mathcal{M}$.
Moreover, let us define
\begin{equation}\label{eq:calpha}
c_\alpha := \inf_{\Bcal_\alpha}\Ecal,\qquad \hat c_\alpha := \inf_{\Ucal_\alpha}\Ecal.
\end{equation}
Being $\Bcal_\alpha$ weakly closed in $\Mcal$, in the Sobolev subcritical case the level $c_\alpha$ is achieved for any $\alpha\geq \lambda_1(\Omega)$, possibly on $\Ucal_\alpha\subset\Bcal_\alpha$. Therefore, in order to find a solution of \eqref{eq:system_elliptic}, it is sufficient to find $\alpha$ such that $c_\alpha < \hat c_\alpha$ (and this will be our strategy).

On the contrary, in the Sobolev critical case, it is also an issue to prove that $c_\alpha$ is achieved: indeed, since $H^1_0(\Omega)$ is not compactly embedded in $L^{2^*}(\Omega)$, $\Ecal|_{\Mcal}$ is no longer weakly lower semicontinuous. To overcome this difficulty, in the spirit of the celebrated paper by Brezis and Nirenberg \cite{MR709644}, we are able to recover the compactness of the minimizing sequences associated to $c_\alpha$ by imposing a bound on the masses $\rho_1,\rho_2$ and on $\alpha$. More precisely, we have the following key result.
\begin{proposition}\label{prop:compact_intro}
%Let $\Omega\subset\R^N$ be a bounded Lipschitz domain, $\mu_1,\mu_2>0$, $\beta\in\R$,
Consider $N\ge3$ and $p=2^*-1$. Suppose that $\rho_1,\rho_2\geq 0$ and $\alpha\geq \lambda_1(\Omega)$ are such that
\begin{equation}\label{eq:compact_intro}
(\rho_1+\rho_2)(\alpha - \lambda_1(\Omega)) \le \frac{1}{\dd^{(N-2)/2}},
\end{equation}
where
\begin{equation}\label{eq:def_Lambda}
\Lambda:=\frac{2S_{N}}{2^*} \max_{\{x^2+y^2 =1\}}
\left(\mu_1 |x|^{2^*} + \mu_2 |y|^{2^*}
+2\beta^+  |xy|^{2^*/2} \right).
\end{equation}
Then any minimizing sequence associated to $c_\alpha$ is relatively compact in $\Bcal_\alpha$. In particular, $c_\alpha$ is achieved.
\end{proposition}
Based on the previous proposition, we introduce the following set of admissible masses
\begin{equation}\label{eq:defA}
A := \left\{(\rho_1,\rho_2)\in [0,\infty)^2 : \begin{array}{l}
\text{$c_\alpha<\hat c_\alpha$ for some $\alpha\ge \lambda_1(\Omega)$,}\smallskip\\
\text{with $\alpha$ satisfying \eqref{eq:compact_intro} if $p=2^*-1$}
\end{array}
\right\}\cup\left\{(0,0)\right\}.
\end{equation}
Notice that, as a matter of fact, $A$ depends on $\Omega$, $N$, $p$, $\mu_1$, $\mu_2$ and $\beta$.
Moreover, if $(\rho_1,\rho_2)\in A$, then we can choose the local minimizer $(u_1,u_2)\in \Mcal$ to be a non-negative solution of
\eqref{eq:system_elliptic} for some $(\omega_1,\omega_2)\in \R^2$. Finally,we introduce the exponents $a$ and $r$
as
%\begin{equation*}
%a = a(N,p):= \frac{N(p-1)}{4}\in \left(1,\frac{N}{N-2}\right] ,\qquad r = r(N,p) := \frac{p+1}{4} - \frac{N(p-1)}{8}\in \left[0,\frac{1}{N}\right).
%\end{equation*}
\begin{equation}\label{eq:newexponent}
a = a(N,p):= \frac{N(p-1)}{4} ,\qquad r = r(N,p) := \frac{p+1}{4} - \frac{N(p-1)}{8}.
\end{equation}
Notice that these two constants appear naturally in this context because (up to a suitable multiple) they enter in the Gagliardo-Nirenberg inequality.
Observe also that
\begin{equation}\label{eq:sign_a-1}
0<a<1 \text{ if }1<p<1+4/N;
\quad a=1 \text{ if }p=1+4/N;
\quad a>1 \text{ if }p=1+4/N.
\end{equation}
\begin{theorem}[$L^2$--supercritical cases: existence]
\label{prop:supercritical}
%Let $\Omega\subset\R^N$ be a bounded Lipschitz domain, $\mu_1,\mu_2>0$, $\beta\in\R$,
Let $1+4/N < p \le 2^*-1$. If $A$ is defined as in \eqref{eq:defA}, then
\[
\text{$A$ is star-shaped with respect to $(0,0)$.}
\]
Moreover, there exists a positive constant $R=R(\Omega,N,p)$ such that if
$\rho_1,\rho_2\ge0$ are such that
\begin{equation}\label{eq:assnice}
\left[\max\{ \mu_1 \rho_1^{2r},\mu_2 \rho_2^{2r}\} +
\beta^+ \rho_1^r\rho_2^{r}  \right]\cdot (\rho_1+\rho_2)^{a-1} \le
R(\Omega,N,p),
\end{equation}
then $(\rho_1,\rho_2)\in A$. Here $a$ and $r$ are defined as in \eqref{eq:newexponent} and $R$ is explicit (see \eqref{eq:def_R} ahead).
\end{theorem}
%
%By the way,
%\[
%R(\Omega,N,p)=\frac{p+1}{2C_{N,p}}\frac{(a-1)^{(a-1)}}{a^a} \lambda_j(\Omega)^{-(a-1)}
%\]
%
\begin{remark}
When $N\ge3$ and $p=2^*-1$, explicit calculations show that $a=N/(N-2)$, $r=0$ and \eqref{eq:assnice}
rewrites as
\[
\rho_1+\rho_2 \le \left[\frac{R(\Omega,N,2^*-1)}{\max\{ \mu_1 ,\mu_2 \} +
\beta^+}\right]^{(N-2)/2}.
\]
\end{remark}
\begin{remark}\label{rem:uniform_beta2}
When $\beta\leq 0$, condition \eqref{eq:assnice} is independent from $\beta$ and reduces to:
\[
\left[\max\{ \mu_1 \rho_1^{2r},\mu_2 \rho_2^{2r}\}  \right]\cdot (\rho_1+\rho_2)^{a-1} \le R(\Omega,N,p),
\]
\end{remark}
\begin{remark}
As we mentioned before, when $p>1+4/N$ the functional $\Ecal$ is unbounded from below on $\Mcal$. We deduce that,
under the same assumptions of Theorem \ref{prop:supercritical}, $\Ecal$ has a mountain pass geometry on $\Mcal$ (see for instance \cite[Thm. 4.2, Ch. II]{Struwebook}).
In a standard way, this implies that
\begin{quote}
if $1+\dfrac{4}{N} < p < 2^*-1$ and $(\rho_1,\rho_2)$ satisfies \eqref{eq:assnice}, then $\left.\Ecal\right|_{\Mcal}$ has
two critical points: one local minimum and one mountain pass.
\end{quote}
We cannot obtain the same result for $p=2^*-1$, since our compactness result in
Proposition \ref{prop:compact_intro} holds only for minimizing sequences.
\end{remark}
\begin{remark}
It is natural to expect that the set $A$ is bounded in $\R^2_+$. Actually,
we know from \cite{ntvAnPDE} that this is the case, for the single equation, at least in the Sobolev
subcritical case. The proof of this fact should follow by a careful blow-up analysis based
on suitable pointwise a priori controls, along the lines of \cite[Section 4]{ntvAnPDE}, and will be the object of further investigation.
\end{remark}
Since the solutions we found in the $L^2$-supercritical cases are local
minima of the energy, it is natural to expect that they correspond to orbitally stable solitary waves. The proof of this fact requires some modification of the standard arguments
used for global minimizers. Notably, the lack of compactness of the embedding $H^1_0 \hookrightarrow L^{2^*}$ is an issue here, too. We have the following result.
\begin{theorem}[$L^2$--supercritical cases: stability]\label{thm:stab}
Let $1+4/N < p \le 2^*-1$ and $(\rho_1,\rho_2)\in A$. Let  $\bar\alpha\ge \lambda_1(\Omega)$ be such that
\[
c_{\bar\alpha}<\hat c_{\bar\alpha},
\qquad
\text{and $\bar\alpha$ satisfies \eqref{eq:compact_intro} if $p=2^*-1$}.
\]
Then the set of \emph{local} ground states
\begin{equation}\label{eq:GroundStates}
G_{\bar\alpha}:=\left\{ (u_1,u_2) \in H^1_0(\Omega;\C^2): \ (|u_1|,|u_2|)\in\Bcal_{\bar \alpha},\ \Ecal(u_1,u_2)=
c_{\bar\alpha} \right\},
\end{equation}
is (conditionally) orbitally stable.
\end{theorem}
As we noticed, in the Sobolev critical case our results are new also for the single equation. In particular, choosing $\rho_2=0$, $\rho_1=\rho$, $\mu_1=\mu$, we have the following direct consequence.
\begin{theorem}
Let $\mu>0$. If
\[
0<\rho\le \left[\frac{R(\Omega,N,2^*-1)}{\mu}\right]^{(N-2)/2}
\]
then the problem
\[
\begin{cases}
-\Delta u + \omega u=\mu u|u|^{2^*-2}\\
\int_\Omega u^2=\rho,\quad u\in H^1_0(\Omega)
\end{cases}
\]
admits a positive solution $u$, for a suitable $\omega\in (-\lambda_1(\Omega),0)$, which is a local minimizer of the associated energy. Moreover, the corresponding set of local ground states is (conditionally) orbitally stable.
\end{theorem}
To conclude, we remark that all the assumptions in our results involve $\beta^+$, the positive part of $\beta$. As a consequence, all our estimates are uniform in $\beta<0$. To summarize, recalling Theorems
\ref{prop:subcritical} and  \ref{prop:supercritical} (see also and Remarks \ref{rem:uniform_beta1} and \ref{rem:uniform_beta2}), we prove existence of solutions whenever $\beta<0$ and $(\rho_1,\rho_2)$ satisfy
\begin{equation}\label{eq:uniform_beta}
\begin{cases}
\rho_1,\rho_2>0 & \quad \text{ if }1<p<1+4/N, \\ \smallbreak
0<\mu_1\rho_1^\frac{2}{N},\mu_2\rho_2^\frac{2}{N}<\frac{N+2}{NC_{N}}  &\quad \text{ if } p=1+\frac{4}{N}\\ \smallbreak
\max\{ \mu_1 \rho_1^{2r},\mu_2 \rho_2^{2r}\}  \cdot (\rho_1+\rho_2)^{a-1} \le R(\Omega,N,p) &\quad  \text{ if }1+\frac{4}{N}<p\leq 2^*-1.
\end{cases}
\end{equation}
This allows to exploit results in \cite{MR2599456,SZ15,STTZ} in order to perform a segregation analysis as $\beta\to-\infty$. %Such analysis is contained in Section \ref{sec:segregation}. We just state here the following consequence of our arguments.
\begin{theorem}\label{thm_beta-infty}
Let $\Omega\subset\R^N$ be a smooth bounded domain, $\mu_1,\mu_2>0$, $\beta< 0$ and $(\rho_1,\rho_2)$ be such that \eqref{eq:uniform_beta} holds. Let also $(u_{1,\beta},u_{2,\beta})$ be a corresponding ground state of \eqref{eq:system_elliptic}, with multipliers $(\omega_{1,\beta},\omega_{2,\beta})$ and such that $u_{1,\beta},u_{2,\beta}>0$ in $\Omega$. Then $\{(u_{1,\beta},u_{2,\beta})\}_{\beta<0}$ is uniformly bounded in $C^{0,\alpha}(\overline \Omega)$ and, up to subsequences, $(u_{1,\beta},u_{2,\beta})\to(w^+,w^-)$ as $\beta\to-\infty$, in $C^{0,\alpha}(\overline{\Omega})\cap H^1_0(\Omega)$, where  $w\in C^{0,1}(\overline \Omega)$ solves
\[
\begin{cases}
-\Delta w + \omega_1 w^+-\omega_2 w^- = \mu_1 (w^+)^p-\mu_2 (w^-)^p & \text{in } \Omega,\\
\int_\Omega (w^+)^2=\rho_1,\ \int_\Omega (w^-)^2=\rho_2,\quad w\in H^1_0(\Omega),
\end{cases}
\]
for $\omega_i:=\lim_{\beta\to -\infty}\omega_{i,\beta}$.
\end{theorem}

\medbreak

The paper is structured as follows. In the next subsection we make some preliminary remarks and definitions which will be used in the text; in particular, we recall some facts about the Gagliardo-Nirenberg inequality and deduce some direct consequences.

Section \ref{sec:subcrit_crit} is devoted to the existence results under (H1)-(H2), i.e., to the proof of Theorem \ref{prop:subcritical}-a) as well as to a detailed explanation of condition \eqref{N=2condition2} (which leads to Fig. \ref{fig:N=2}).

The existence results under (H3)-(H4) (Theorem \ref{prop:supercritical}) are proved in Section \ref{sec:3}. Therein, we provide lower estimates for $\hat c_\alpha$ (see Subsection \ref{sec:hat_c_below}), we prove a slightly more general version of Proposition \ref{prop:compact_intro} (Subsection \ref{sec:calpha_achiev}), while in Subsection \ref{sec:abstractexist}  we introduce an abstract criterium  that guarantees that $(\rho_1,\rho_2)\in A$. Finally, Subsections \ref{sec:star} and \ref{sec:explicit_est} contain respectively the proofs of the  qualitative properties of $A$ and the deduction of condition \eqref{eq:assnice}.

Section \ref{sec:stab} is concerned with the proof of the stability results, namely the proof of  Theorems \ref{prop:subcritical}-b) and \ref{thm:stab}. Finally, Theorem \ref{thm_beta-infty} is proved in Section \ref{sec:segregation}.

\subsection{Notations and Preliminaries}\label{eq:preliminaries} Throughout the paper we denote by $\lambda_1(\Omega)$ the first eigenvalue of the Dirichlet Laplacian in $\Omega$, and by $\varphi_1$ the corresponding first eigenfunction, which we assume  normalized in $L^2(\Omega)$ and positive in $\Omega$.

We use the following $L^q(\Omega)$ ($1\le q<\infty$) and $H^1_0(\Omega)$--norms:
\[
\|u\|_{L^q(\Omega)}^q:=\int_\Omega |u|^q,\qquad  \|u\|_{H^1_0(\Omega)}^2:= \int_\Omega |\nabla u|^2.
\]
Where there is no risk of confusion, we will denote $\|\cdot\|_{L^q(\Omega)}$ simply by $\|\cdot \|_q$.

The Gagliardo-Nirenberg inequality asserts that there exists a constant $C_{N,p}$ such that
\begin{equation}\label{eq:GN}
\begin{split}
\|v\|_{L^{p+1}(\R^N)}^{p+1}&\leq C_{N,p} \|\nabla v\|_{L^2(\R^N)}^{N(p-1)/2}\|v\|_{L^2(\R^N)}^{p+1-{N(p-1)/2}}\\
&= C_{N,p} \|\nabla v\|_{L^2(\R^N)}^{2a}\|v\|_{L^2(\R^N)}^{4r},\qquad \forall v\in H^1(\R^N),
\end{split}
\end{equation}
where the exponents $a$ and $r$ were defined in \eqref{eq:newexponent}. We remark that this inequality holds also in $H^1_0(\Omega)$, for any bounded domain $\Omega$,
with the same constant $C_{N,p}$.
It is proved in \cite{Weinstein1983} that
\begin{equation}\label{eq:CN4}
C_{N,p}:=\inf_{v\in H^1(\R^N)\setminus\{0\}} \frac{\|v\|_{L^{p+1}(\R^N)}^{p+1}}{\|\nabla v\|_{L^2(\R^N)}^{2a}  \|v\|_{L^2(\R^N)}^{4r}    } = \frac{\|Z\|_{L^{p+1}(\R^N)}^{p+1}}{\|\nabla Z\|_{L^2(\R^N)}^{2a}  \|Z\|_{L^2(\R^N)}^{4r}    },
\end{equation}
where $Z$ is, up to translations, the unique (see \cite{Kwong1989}) positive solution of
\begin{equation}\label{eq:Kwong}
-\Delta Z+Z=Z^p, \qquad Z\in H^1(\R^N).
\end{equation}
In particular, the inequality on $H^1_0(\Omega)$ is strict unless $v$ is trivial.
In the special case $p=1+{4/N}$ we  denote
\begin{equation}\label{eq:def_CN}
C_N:=C_{N,1+{4/N}},
\end{equation}
while for $p={(N+2)/(N-2)}$ and $N\geq 3$,
\begin{equation}\label{eq:Sobolev_constant}
S_N:=C_{N,(N+2)/(N-2)}.
\end{equation}
Observe that $S_N$ is just the best Sobolev constant of the embedding $\mathcal{D}^{1,2}(\R^N)\hookrightarrow L^{2N/(N-2)}(\R^N)$:
\[
\|v\|_{L^{2N/(N-2)}(\R^N)}^{2N/N-2}\leq S_N \|\nabla v\|_{L^2(\R^N)}^{2N/(N-2)},  \qquad \forall v\in \mathcal{D}^{1,2}(\R^N)
\]
For $(u_1,u_2)\in \Mcal$, defined as in \eqref{eq:defM}, using the H\"older and Gagliardo-Nirenberg inequalities (on bounded domains) we have
\begin{equation}\label{eq:importantestimateG}
\begin{split}
\int_\Omega \mu_1 |u_1|^{p+1} + &2\beta |u_1|^{(p+1)/2} |u_2|^{(p+1)/2} + \mu_2 |u_2|^{p+1}\\
&\le \mu_1\|u_1\|_{p+1}^{p+1}+\mu_2\|u_2\|_{p+1}^{p+1}+2\beta^+ \|u_1\|_{p+1}^{(p+1)/2}\|u_2\|_{p+1}^{(p+1)/2}\\
&< C_{N,p}\left(\mu_1 \rho_1^{2r}\|\nabla u_1\|_2^{2a} +\mu_2 \rho_2^{2r}\|\nabla u_2\|_2^{2a} +2\beta^+ \rho_1^r\rho_2^{r}  \|\nabla u_1\|_2^{a} \| \nabla u_2\|_2^{a}\right).
\end{split}
\end{equation}
where the exponents $a$ and $r$ are defined in \eqref{eq:newexponent}. As a consequence we have, for $(u_1,u_2)\in \Mcal$,
\begin{multline}\label{eq:importantestimate}
\Ecal(u_1,u_2)
			>
\frac{1}{2}(\|\nabla u_1\|_2^2 + \|\nabla u_2\|_2^2) \\
-\frac{C_{N,p}}{p+1}\left(\mu_1 \rho_1^{2r}\|\nabla u_1\|_2^{2a}
+\mu_2 \rho_2^{2r}\|\nabla u_2\|_2^{2a}
+2\beta^+ \rho_1^r\rho_2^{r}  \|\nabla u_1\|_2^{a} \| \nabla u_2\|_2^{a}\right).
\end{multline}
According to \eqref{eq:sign_a-1} and to the previous inequality, the $L^2$-critical value $p=1+4/N$ is the threshold for the coercivity of $\Ecal$ over $\Mcal$, as we shall see more in detail in the following.

%%%
%%SECTION SUBCRIT and CRIT

\section{The \texorpdfstring{$L^2$}{L\texttwosuperior}--subcritical and \texorpdfstring{$L^2$}{L\texttwosuperior}--critical cases}\label{sec:subcrit_crit}

In this section we deal with conditions (H1) and (H2), meaning that
\begin{equation}\label{H1}
1<p\leq 1+\frac{4}{N}.
\end{equation}
Recall the definition of $C_N$ in \eqref{eq:def_CN}.

\begin{proof}[Proof of Theorem \ref{prop:subcritical}-a)]
Let us show in the two cases that $\Ecal$ restricted to $\Mcal$ is coercive. Then, by the direct method of the calculus of variations, $\inf_{\Mcal}\Ecal$ is achieved by a couple $(u_1,u_2)$ (which belongs to $\Mcal$ because of the compact embedding $H^1_0(\Omega)\hookrightarrow L^2(\Omega)$). By the Lagrange multipliers rule, $(u_1,u_2)$ solves \eqref{eq:system_elliptic} for some $(\omega_1,\omega_2)\in\R^2$. By possibly taking $|u_i|$, we can suppose $u_i\geq0$ and, if $\rho_1,\rho_2>0$, the maximum principle provides $u_i>0$ (indeed, since $\Omega$ is Lipschitz, each $u_i$ is continuous up to the boundary).

If (H1) holds, then $0<{N(p-1)/2}<2$, so that, in \eqref{eq:importantestimate}, $a<1$; we immediately deduce that $\Ecal$ restricted to $\Mcal$ is coercive for every $\rho_1,\rho_2>0$.

In case we have (H2),  continuing from \eqref{eq:importantestimate} and since in this case $a=1$, $r={1/N}$, we have 
\begin{align}
\Ecal(u_1,u_2)	&> \frac{1}{2}(\|\nabla u_1\|_2^2 + \|\nabla u_2\|_2^2) \\
                         & \qquad -\frac{NC_N}{2(N+2)}\left(\mu_1 \rho_1^{2/N}  \|\nabla u_1\|_2^2 +\mu_2 \rho_2^{2/N} \|\nabla u_2\|_2^2+2\beta^+ (\rho_1 \rho_2)^{1/N}\|\nabla u_1\|_2\| \nabla u_2\|_2\right) \\
			&= \frac{1}{2}\|\nabla u_1\|_2^2\left(1-\frac{NC_N\mu_1\rho_1^{2/N}}{N+2}\right) + \frac{1}{2}\|\nabla u_2\|_2^2\left(1-\frac{NC_N\mu_2\rho_2^{2/N}}{N+2}\right)\\
			& \qquad -\frac{\beta^+ N C_N(\rho_1\rho_2)^{1/N}}{N+2}\|\nabla u_1\|_2\|\nabla u_2\|_2\\
			&= \frac{1}{2} \begin{bmatrix} \|\nabla u_1\|_2 & \|\nabla u_2\|_2 \end{bmatrix} \cdot A \cdot \begin{bmatrix} \|\nabla u_1\|_2 & \|\nabla u_2\|_2 \end{bmatrix}^T  \label{eq:L^2crit_energyestimate},
\end{align}
where
\[
A=\begin{bmatrix}
 1-\frac{NC_N\mu_1\rho_1^{2/N}}{N+2}  & -\frac{\beta^+ N C_N(\rho_1\rho_2)^{1/N}}{N+2} \\
-\frac{\beta^+ N C_N(\rho_1\rho_2)^{1/N}}{N+2} &  1-\frac{NC_N\mu_2\rho_2^{2/N}}{N+2}.
\end{bmatrix}
\]
If $A$ is positive definite then our result follows. Now $A$ is positive definite if and only if the following inequalities are simultaneously satisfied
\begin{align*}
&1-\frac{NC_N\mu_1\rho_1^{2/N}}{N+2}>0, \qquad 1-\frac{NC_N\mu_2\rho_2^{2/N}}{N+2}>0, \\
&\left(1-\frac{NC_N\mu_1\rho_1^{2/N}}{N+2}\right)\left(1-\frac{NC_N\mu_2\rho_2^{2/N}}{N+2}\right)-\left(\frac{N C_N}{N+2}\right)^2(\beta^+)^2 (\rho_1 \rho_2)^{2/N}>0,
\end{align*}
that is to say \eqref{N=2condition2} holds.
\end{proof}

\begin{remark}\label{rem:N=2}
In this remark we interpret in the $(\mu_1\rho_1^{2/N},\mu_2\rho_2^{2/N})$--plane  the condition \eqref{N=2condition2} (see Fig. \ref{fig:N=2} for a visualization of this remark). Let $\bar x=\mu_1\rho_1^{2/N}$, $\bar y=\mu_2\rho_2^{2/N}$ so that \eqref{N=2condition2} corresponds to $(\bar x,\bar y)\in C$, where
\[
C:=\left\{(x,y)\in \R^2:\ 0< x,y< \frac{N+2}{NC_{N}}\ \text{ and }\ x+y+\frac{NC_N}{N+2} \frac{(\beta^+)^2-\mu_1\mu_2}{\mu_1\mu_2}xy<\frac{N+2}{NC_{N}} \right\}
\]
For $\beta\leq 0$, the condition reduces to $(\bar x,
\bar y)$ lying in the square
\[
Q:=\left\{ (x,y)\in \R^2:\ 0< x,y< \frac{N+2}{NC_{N}}\right\}
\]
For $\beta=\sqrt{\mu_1\mu_2}$, we have a half-square:
\[
Q_1:=Q\cap \left\{(x,y)\in \R^2:\ x+y< \frac{N+2}{NC_{N}}\right\}.
\]
For $\beta>0$, $\beta\neq \sqrt{\mu_1\mu_2}$, the curve
\[
\left\{(x,y)\in \R^2:\ x+y+\frac{NC_N}{N+2} \frac{(\beta^+)^2-\mu_1\mu_2}{\mu_1\mu_2}xy=\frac{N+2}{NC_{N}}  \right\}
\]
is an hyperbola which contains the points $(0,\frac{N+2}{NC_{N}})$, $(\frac{N+2}{NC_{N}},0)$. This hyperbola is the graph of
\[
y=\left(\frac{N+2}{NC_{N}}-x\right)\left(1+\frac{NC_N}{N+2}\frac{(\beta^+)^2-\mu_1\mu_2}{\mu_1\mu_2} x\right)^{-1}
\]
or, equivalently,
\[
y=\frac{N+2}{NC_{N}}\left(-\frac{\mu_1\mu_2}{(\beta^+)^2-\mu_1\mu_2}+\frac{N+2}{NC_{N}} \frac{(\beta^+)^2}{(\beta^+)^2-\mu_1\mu_2}\left(\frac{N+2}{NC_{N}}+\frac{(\beta^+)^2-\mu_1\mu_2}{\mu_1\mu_2}x\right)^{-1}  \right)
\]
which has a vertical asymptote at $x= \frac{N+2}{NC_{N}}\frac{\mu_1\mu_2}{\mu_1\mu_2-(\beta^+)^2}$.
Thus, the set $C$ always contains the sides of the square $Ox^+\cap \overline{Q}$ and $Oy^+\cap \overline{Q}$. When $0<\beta<\sqrt{\mu_1\mu_2}$ it contains $Q_1$, and when $\beta>\sqrt{\mu_1\mu_2}$ it is contained in $Q_1$.
\end{remark}
\begin{remark}
When $\beta\leq 0$, the condition reads as
\[
\mu_1 \rho_1^{2/N}, \mu_2 \rho_2^{2/N}<\frac{N+2}{NC_{N}}.
\]
Going to \cite[p. 1833]{ntvAnPDE}  we see that, as a consequence of Pohozaev's identity:
\[
\frac{N+2}{NC_{N}}=\|Z\|_2^{2/N},
\]
with $Z$ defined in \eqref{eq:Kwong}. Therefore the condition is equivalent to
\[
\rho_1<\|Z\|_2 \mu_1^{-N/2},\qquad \rho_2<\|Z\|_2 \mu_2^{-N/2}.
\]
This is consistent with the results in \cite{FibichMerle2001,ntvAnPDE}, which correspond to the case $\beta=0$ in \eqref{eq:system_elliptic}.
\end{remark}

\section{The \texorpdfstring{$L^2$}{L\texttwosuperior}--supercritical and Sobolev--subcritical case. The Sobolev--critical case}\label{sec:3}

\subsection{Preliminaries}

Assume from now on that $p$ satisfies either (H3) or (H4), that is $p>1+4/N$, with $p\leq (N+2)/(N-2)$
if $N\geq 3$. Along this section we do not make any distinction between the Sobolev-critical and the
Sobolev-subcritical cases, unless otherwise specified.

In Proposition \ref{prop:subcritical} we proved that $\Ecal$ restricted to $\Mcal$ is coercive for any $\rho_1,\rho_2>0$ under (H1) or for $(\rho_1,\rho_2)$ satisfying \eqref{N=2condition2}  under (H2). Thus, solutions were found as global minimizers of $\Ecal|_\Mcal$.
In the $L^2$-supercritical case $p>1+4/N$ the previous approach cannot work, since $\Ecal$ restricted to $\Mcal$ is not coercive for every value of $(\rho_1,\rho_2)$, as we  show in the following lemma.
Notice that this was already suggested by equation \eqref{eq:importantestimate}, since now $a=N(p-1)/4>1$.

\begin{lemma}\label{lemma:geometry_mp}
Let $p>1+4/N$. Then there exists $(U_{1,k},U_{2,k})\in \Mcal$, with nonnegative components, such that, as $k\to \infty$,
\[
\|(U_{1,k},U_{2,k})\|_{H^1_0(\Omega)}\to +\infty \quad \text{ and }\qquad \Ecal(U_{1,k},U_{2,k})\to -\infty.
\]
\end{lemma}
\begin{proof}
Let $\phi\in C^\infty_c(B_1)$ with $\phi>0$ in $B_1$ and $\int_{B_1} \phi^2=1$, and $x_1,x_2\in\Omega$ such that $x_1\neq x_2$. For $k\in \N$ and $i=1,2$, we define
\[
U_{i,k}(x)=\rho_i^{1/2} k^{N/2} \phi(k(x-x_i)), \qquad x\in \Omega.%\in B_\frac{1}{n}(x_i).
\]
For $k$ sufficiently large we have $\textrm{supp}(U_{i,k})\subset B_{1/k}(x_i)\subset\Omega$, $i=1,2$, and $\textrm{supp}(U_{1,k})\cap \textrm{supp}(U_{2,k})=\emptyset$. Furthermore,
\[
\int_\Omega U_{i,k}^2=\rho_i \int_{B_1} \phi^2 =\rho_i,
\]
so that $(U_{1,k},U_{2,k}) \in \Mcal$ for $k$ sufficiently large. We compute
\[
\|(U_{1,k},U_{2,k})\|_{H^1_0(\Omega)} = k \sqrt{\rho_1+\rho_2} \|\nabla\phi\|_{L^2(B_1)}
\to+\infty
\]
as $k\to+\infty$ and
\[
\Ecal(U_{1,n},U_{2,n})=
k^2 \frac{\rho_1+\rho_2}{2}\|\nabla \phi\|_{L^2(B_1)}^2
- k^{2a} \frac{\mu_1\rho_1^{(p+1)/2}+\mu_2\rho_2^{(p+1)/2}}{p+1} \|\phi\|_{L^{p+1}(B_1)}^{p+1}
\to -\infty
\]
as $k\to+\infty$, since $a>1$.
\end{proof}

\subsection{A basic estimate on \texorpdfstring{$\hat c_\alpha$}{hat c alpha}}\label{sec:hat_c_below}

In order to prove the existence of a solution of \eqref{eq:system_elliptic} under (H3) or (H4) for
certain values of $\rho_1,\rho_2$, we use a different approach than the one used in Section
\ref{sec:subcrit_crit}. Recall that, for $\alpha\geq \lambda_1(\Omega)$, $\Bcal_\alpha$ and
$\Ucal_\alpha$ are defined in \eqref{eq:Balpha}, while $c_\alpha$ and $\hat c_\alpha$ are
as in \eqref{eq:calpha}. Observe that $\Bcal_\alpha\neq \emptyset$, since it contains at least
$(\sqrt{\rho_1}\varphi_1,\sqrt{\rho_2}\varphi_1)$. Moreover
\[
c_{\lambda_1(\Omega)}=\hat c_{\lambda_1(\Omega)}=\Ecal(\sqrt{\rho_1}\varphi_1,\sqrt{\rho_2}\varphi_1).
\]

Recalling \eqref{eq:importantestimate} and using the identification $x=\|\nabla u_1\|_2$, $y=\|\nabla u_2\|_2$, we end up studying the function $\varphi:\R^2_+\to \R$ defined by
\begin{equation}\label{eq:vphi_def}
\Phi(x,y)=\frac{1}{2}(x^2+y^2)-\frac{C_{N,p}}{p+1}\left(\mu_1 \rho_1^{2r}x^{2a}
+\mu_2 \rho_2^{2r}y^{2a}
+2\beta^+ \rho_1^r\rho_2^{r}  x^{a} y^{a}\right)
\end{equation}
where now $a>1$. Indeed, by \eqref{eq:importantestimate} we obtain that
\begin{equation}\label{eq:Phi_ineq}
\Ecal(u_1,u_2)\geq \Phi \left(\|\nabla u_1\|_2,\|\nabla u_1\|_2\right)\qquad\text{ for every }(u_1,u_2)\in \Mcal.
\end{equation}
In particular, this allows to estimate $\hat c_\alpha$ from below. To do that, let us define the following subsets of $\R^2$:
\[
U_\alpha=\left\{(x,y)\in \R^2_+:\ x^2+y^2= (\rho_1+\rho_2)\alpha\right\},\quad V_\alpha=U_\alpha\cap \left\{x\geq \sqrt{\rho_1\lambda_1(\Omega)},\ y\geq \sqrt{\rho_2\lambda_1(\Omega)}\right\}.
\]
The set $U_\alpha$ is obtained from $\Ucal_\alpha$ through the identification $x=\|\nabla u_1\|_2$, $y=\|\nabla u_2\|_2$. The set $V_\alpha$ is motivated by the fact that, for $(u_1,u_2)\in \Mcal$,
$\|\nabla u_1\|^2\geq \rho_1\lambda_1(\Omega)$ and $\|\nabla u_2\|^2\geq \rho_2\lambda_1(\Omega)$.
With this notation, using \eqref{eq:Phi_ineq}, we obtain
\begin{equation}\label{eq:iniziostima}
\hat c_\alpha \ge \min_{(x,y)\in V_\alpha} \Phi(x,y) = \frac12 \alpha
-\frac{C_{N,p}}{p+1} \max_{(x,y)\in V_\alpha} \left(\mu_1 \rho_1^{2r}x^{2a}
+\mu_2 \rho_2^{2r}y^{2a}
+2\beta^+ \rho_1^r\rho_2^{r}  x^{a} y^{a}\right).
\end{equation}
Now, due to the limitations in the definition of $V_\alpha$, the last maximum can not be written
explicitly in terms of $\alpha$ (except for a few particular cases). For this reason, we prefer the
more rough estimate in which $V_\alpha$ is replaced with $U_\alpha$. This allows more readable
results, without modifying the qualitative structure of the estimates.
\begin{lemma}\label{lem:hatc_from_below}
Let
\begin{equation}\label{eq:defdd}
\dd=\dd(\rho_1,\rho_2):= \frac{2C_{N,p}}{p+1} \max_{t\in [0,\pi/2]}
\left(\mu_1 \rho_1^{2r}\cos^{2a}t
+\mu_2 \rho_2^{2r}\sin^{2a}t
+2\beta^+ \rho_1^r\rho_2^{r}  \cos^{a}t \sin^{a}t\right).
\end{equation}
Then, for every $\alpha > \lambda_1(\Omega)$,
\[
\hat c_\alpha > \frac12\left((\rho_1+\rho_2)\alpha - \dd(\rho_1,\rho_2)(\rho_1+\rho_2)^a \alpha^a \right).
\]
\end{lemma}
\begin{proof}
Since
\[
V_\alpha\subset U_\alpha = \left\{(\cos t,\sin t)\sqrt{(\rho_1+\rho_2) \alpha}:t\in[0,\pi/2]\right\},
\]
the lemma follows by continuing the estimate in
\eqref{eq:iniziostima}.
\end{proof}
\begin{remark}
Notice that $\dd$ depends on $\mu_1,\mu_2,\beta$, and also on $p$ and $N$ (via $a$ and $r$).
On the other hand, in case $N\ge 3$ and $p=2^*-1$, we have that $a=2^*/2$, $r=0$ and $\dd$ does not
depend on $\rho_1,\rho_2$, and actually its definition coincides with that given in
\eqref{eq:def_Lambda}. Then we have, for any $(v_1,v_2)\in H^1_0(\Omega;\R^2)$,
\begin{equation}\label{eq:lambda_max_def}
\frac{2 S_N}{2^*} \left(\mu_1\|\nabla v_1\|_2^{2^*}+2\beta^+ \|\nabla v_1\|_2^{2^*/2}\|\nabla v_{2}
\|_2^{2^*/2}+\mu_2\|\nabla v_{2}\|_2^{2^*}\right) \le \dd \left(\|\nabla v_1\|_2^{2}+
\|\nabla v_{2}\|_2^2\right)^{2^*/2}
\end{equation}
(recall the definition of $S_N = C_{N,2^*-1}$ given in \eqref{eq:Sobolev_constant}).
To see this,  we notice that for any $(v_1,v_2)$ one can find $t\in[0,\pi/2]$ such that
\[
\|\nabla v_1\|_2 = \left(\|\nabla v_1\|_2^{2}+
\|\nabla v_{2}\|_2^2\right)^{1/2}\cos t,\qquad
\|\nabla v_2\|_2 = \left(\|\nabla v_1\|_2^{2}+
\|\nabla v_{2}\|_2^2\right)^{1/2}\sin t,
\]
and we substitute in \eqref{eq:defdd}.
\end{remark}

\subsection{The level \texorpdfstring{$c_\alpha$}{c alpha} is achieved}\label{sec:calpha_achiev}

As we mentioned, we will look for local minimizers of $\Ecal$ on $\Bcal_\alpha$, hence at level
$c_\alpha$, for suitable values of $\alpha$. A first necessary step is to prove that $c_\alpha$ is achieved
(possibly on $\Ucal_\alpha$, the topological boundary of $\Bcal_\alpha$).
This is easily obtained, for every $\alpha\ge \lambda_1(\Omega)$, in the Sobolev subcritical case:
indeed, in such situation, $\Bcal_\alpha$ is weakly compact and $\Ecal$ weakly lower semicontinuous.
On the other hand, if $N\ge3$ and $p=2^*-1$, $\Ecal$ is no longer weakly lower semicontinuous. In this
situation, inspired by the celebrated paper by Brezis and Nirenberg \cite{MR709644}, we can recover
compactness of the minimizing sequences by imposing a smallness condition on the masses
$\rho_1,\rho_2$ and on $\alpha$, as stated in Proposition \ref{prop:compact_intro}. Actually, here we will prove a slightly more general result, considering sequences in which also the masses are not fixed; this will be useful when dealing with stability issues.

In the following, recall that $\dd$ has been introduced in \eqref{eq:defdd} (or, equivalently, in
\eqref{eq:def_Lambda}) and that, in the Sobolev critical case, it does not depend on $\rho_1,\rho_2$.

\begin{proposition}\label{prop:Sobcritconv}
Let $\alpha>\lambda_1(\Omega)$, $\rho_1,\rho_2>0$ satisfy
\[
(\rho_1+\rho_2)(\alpha - \lambda_1(\Omega)) <
\frac{1}{\dd^{(N-2)/2}},
\]
and let $(u_{1,n},u_{2,n})_n$ be such that
\begin{equation}\label{eq:minimizing_seq}
\begin{cases}
\|u_{i,n}\|_2^2=\rho_i + o(1) \qquad & \text{for } i=1,2, \\
\|\nabla u_{1,n}\|_2^2+\|\nabla u_{2,n}\|_2^2 \leq \alpha(\rho_1+\rho_2) + o(1)
\qquad &  \\
c_\alpha \leq \mathcal{E}(u_{1,n},u_{2,n}) \leq c_\alpha +o(1) \quad &
\end{cases}
\end{equation}
as $n\to \infty$. Then, up to subsequences,
\[
(u_{1,n},u_{2,n}) \to (\bar u_{1},\bar u_{2}), \qquad\text{strongly in }H^1_0(\Omega).
\]
In particular, $c_\alpha$ is achieved.
\end{proposition}
\begin{proof}[Proof of Proposition \ref{prop:compact_intro}]
The proposition follows as a particular case of Proposition \ref{prop:Sobcritconv}, when
in  \eqref{eq:minimizing_seq} both $\|u_{i,n}\|_2^2=\rho_i$, $i=1,2$, and
$\|\nabla u_{1,n}\|_2^2+\|\nabla u_{2,n}\|_2^2 \leq \alpha(\rho_1+\rho_2)$.
\end{proof}
\begin{proof}[Proof of Proposition \ref{prop:Sobcritconv}]
By assumption there exists $(\bar u_1,\bar u_2)\in H^1_0(\Omega;\R^2)$ such that, up to subsequences,
\[
\begin{cases}
\|\bar u_i\|_2^2=\rho_i \qquad & \text{for } i=1,2 \\
u_{i,n} \rightharpoonup \bar u_i \qquad & H^1_0(\Omega)\text{-weak } \text{for } i=1,2\\
\|\nabla \bar u_i\|_2^2 \leq \liminf_{n\to\infty} \|\nabla u_{i,n}\|_2^2 \qquad & \text{for } i=1,2.
\end{cases}
\]
Notice that $(\bar u_1,\bar u_2)$ is admissible for the minimization problem $c_\alpha$, whence
\begin{equation}\label{eq:bar_u_i_admissible}
\mathcal{E}(\bar u_1,\bar u_2)\geq c_\alpha.
\end{equation}
Let $v_{i,n}=u_{i,n}-\bar u_i$ and notice that, for $i=1,2$,
\begin{equation}\label{eq:v_i_n}
v_{i,n}\rightharpoonup 0 \quad\text{weakly both } H^1_0(\Omega) \text{ and } L^{2^*}(\Omega),
\qquad v_{i,n}\to0 \quad \text{strongly } L^2(\Omega),
\end{equation}
where $2^*=2N/(N-2)$.

Notice that the strong convergence of a subsequence of $(u_{1,n},u_{2,n})$ is equivalent to the statement:
\begin{equation}\label{eq:BN_easy_case}
\text{ there exists a subsequence } (v_{1,n_k},v_{2,n_k}) \text{ such that }   \|\nabla v_{1,n_k}\|_2^2+\|\nabla v_{2,n_k}\|_2^2  \to 0.
\end{equation}
In such a case, by continuity of the Sobolev embeddings, we have that $c_\alpha=\mathcal{E}(\bar u_1,\bar u_2)$. Since a minimizing sequence for $c_\alpha$ exists, and it satisfies \eqref{eq:minimizing_seq}, we deduce that $c_\alpha$ is achieved.

To conclude the proof, suppose by contradiction that \eqref{eq:BN_easy_case} does not hold, so that
\begin{equation}\label{eq:BN_difficult_case}
\|\nabla v_{1,n}\|_2^2+\|\nabla v_{2,n}\|_2^2 \geq K>0 \quad \text{eventually.}
\end{equation}
We can write
\begin{multline}\label{eq:energy_rewritten}
\mathcal{E}(u_{1,n},u_{2,n}) =
\frac{1}{2} \left(\|\nabla(\bar u_1+v_{1,n})\|_2^2+\|\nabla(\bar u_2+v_{2,n})\|_2^2\right) \\
-\frac{1}{2^*} \left(\mu_1 \|\bar u_1+v_{1,n}\|_{2^*}^{2^*}
+2\beta \|(\bar u_1+v_{1,n})(\bar u_2+v_{2,n})\|_{2^*/2}^{2^*/2}
+\mu_2 \|\bar u_2+v_{2,n}\|_{2^*}^{2^*} \right).
\end{multline}
Notice that, by weak convergence, we have, for $i=1,2$,
\begin{equation}\label{eq:weak_conv}
\|\nabla (\bar u_i+v_{i,n})\|_2^2=\|\nabla \bar u_i\|_2^2+\|\nabla v_{i,n}\|_2^2+o(1)
\qquad \text{as }n\to\infty.
\end{equation}
In order to estimate the remaining terms of \eqref{eq:energy_rewritten}, we recall the following Lemma by  Brezis and Lieb \cite{BL83}: given $1\leq q<\infty$,  if $\{f_n\}_n\subset L^q(\Omega)$ is a sequence bounded in $L^q(\Omega)$, such that $f_n\to f$ almost everywhere, then
\begin{equation}\label{eq:BL}
\|f_n\|_q^q=\|f\|_q^q+\|f_n-f\|_q^q+o(1) \qquad \text{as }n\to\infty.
\end{equation}

\smallbreak

We apply \eqref{eq:BL} first with $f_n=u_{i,n}=\bar u_i+v_{i,n}$ and $q=2^*$ to get
\begin{equation}\label{eq:L4conv}
\|\bar u_i+v_{i,n}\|_{2^*}^{2^*}=\|\bar u_i\|_{2^*}^{2^*}+\|v_{i,n}\|_{2^*}^{2^*}+o(1)
\qquad \text{as }n\to\infty,
\end{equation}
then we apply it with $f_n=(\bar u_1+v_{1,n})(\bar u_2+v_{2,n})$ and $q=2^*/2$ to obtain
\begin{equation}\label{eq:mixed_term_conv1}
 \|(\bar u_1+v_{1,n})(\bar u_2+v_{2,n})\|_{2^*/2}^{2^*/2}
=\|\bar u_1\bar u_2\|_{2^*/2}^{2^*/2}+\|\bar u_1 v_{2,n}+\bar u_2 v_{1,n} +v_{1,n}v_{2,n}\|_{2^*/2}^{2^*/2}+o(1)
\end{equation}
as $n\to\infty$. In order to estimate the second term in the right hand side of \eqref{eq:mixed_term_conv1}, we shall need two inequalities. For every $q>1$ and for every $a,b\in\R$ it holds
\begin{equation}\label{eq:ineq1}
|a+b|^q\leq 2^{q-1} (|a|^q+|b|^q);
\end{equation}
\begin{equation}\label{eq:ineq2}
||a+b|^q-|a|^q| \leq C(|a|^{q-1}|b|+|b|^q),
\end{equation}
for a constant $C$ not depending on $a$ and $b$. By \eqref{eq:ineq1}, we have
\begin{equation}\label{eq:ineq3}
\|\bar u_1 v_{2,n}+\bar u_2 v_{1,n}\|_{2^*/2}^{2^*/2} \leq 2^{(2^*-2)/2}
\left( \|\bar u_1v_{2,n}\|_{2^*/2}^{2^*/2} +\|\bar u_2v_{1,n}\|_{2^*/2}^{2^*/2} \right)
=o(1) \qquad \text{as }n\to\infty,
\end{equation}
because $|v_{i,n}|^{2^*/2}\rightharpoonup 0$ in $L^2(\Omega)$-weak as $n\to+\infty$, for $i=1,2$. Then using \eqref{eq:ineq2}, the H\"older inequality and \eqref{eq:ineq3}, we compute
\begin{multline*}
\left| \|\bar u_1 v_{2,n}+\bar u_2 v_{1,n} +v_{1,n}v_{2,n}\|_{2^*/2}^{2^*/2} -
\|v_{1,n}v_{2,n}\|_{2^*/2}^{2^*/2}\right| \\
\leq C \int_\Omega \left( |v_{1,n}v_{2,n}|^{(2^*-2)/2}|\bar u_1 v_{2,n}+\bar u_2 v_{1,n}|+|\bar u_1 v_{2,n}+\bar u_2 v_{1,n}|^{2^*/2} \right)\,dx\\
\leq C \|v_{1,n}\|_{2^*}^{(2^*-2)/2}\|v_{2,n}\|_{2^*}^{(2^*-2)/2}\|\bar u_1 v_{2,n}+\bar u_2 v_{1,n}\|_{2^*/2}+\|\bar u_1 v_{2,n}+\bar u_2 v_{1,n}\|_{2^*/2}^{2^*/2} =o(1)
\end{multline*}
as $n\to+\infty$. This last estimate, replaced into \eqref{eq:mixed_term_conv1}, provides
\begin{equation}\label{eq:mixed_term_conv}
 \|(\bar u_1+v_{1,n})(\bar u_2+v_{2,n})\|_{2^*/2}^{2^*/2}
=\|\bar u_1\bar u_2\|_{2^*/2}^{2^*/2}+\|v_{1,n}v_{2,n}\|_{2^*/2}^{2^*/2}+o(1)
\end{equation}
as $n\to+\infty$.

By replacing \eqref{eq:weak_conv}, \eqref{eq:L4conv} and \eqref{eq:mixed_term_conv} into \eqref{eq:energy_rewritten}, we see that
\[
\mathcal{E}(u_{1,n},u_{2,n})=\mathcal{E}(\bar u_1,\bar u_2)+\mathcal{E}(v_{1,n},v_{2,n})+o(1) \qquad \text{as }n\to\infty.
\]
The last expression, together with \eqref{eq:minimizing_seq} and \eqref{eq:bar_u_i_admissible}, implies
\[
\mathcal{E}(v_{1,n},v_{2,n})\leq o(1) \qquad \text{as }n\to\infty,
\]
whence, using \eqref{eq:importantestimateG} (with $r=0$, $a=2^*/2$ and $C_{N,2^*-1}=S_N$) and \eqref{eq:lambda_max_def},
\[
\begin{split}
\|\nabla v_{1,n}\|_2^2+\|\nabla v_{2,n}\|_2^2 &\leq
%\frac{2}{2^*}\left(\mu_1\|v_{1,n}\|_{2^*}^{2^*}+2\beta\|v_{1,n}v_{2,n}\|_{2^*/2}^{2^*/2}+\mu_2 \|v_{2,n}\|_{2^*}^{2^*}\right)+o(1) \\
%\leq
%\frac{2}{2^*}\left(\mu_1\|v_{1,n}\|_{2^*}^{2^*}+2\beta^+\|v_{1,n}\|_{2^*}^{2^*/2}\|v_{2,n}\|_{2^*}^{2^*/2}+\mu_2 \|v_{2,n}\|_{2^*}^{2^*}\right)+o(1)
\frac{2 S_N}{2^*} \left(\mu_1\|\nabla v_{1,n}\|_2^{2^*}+2\beta^+ \|\nabla v_{1,n}\|_2^{2^*/2}\|\nabla v_{2,n}\|_2^{2^*/2}+\mu_2\|\nabla v_{2,n}\|_2^{2^*}\right)+o(1)\\
&\leq\dd \left(\|\nabla v_1\|_2^{2}+
\|\nabla v_{2}\|_2^2\right)^{2^*/2}
+o(1).
\end{split}
\]
Now, we can use \eqref{eq:BN_difficult_case} to rewrite the last inequality as
\[
\left(\|\nabla v_{1,n}\|_2^2+\|\nabla v_{2,n}\|_2^2\right)^{(2^*-2)/2}\ge \frac{1}{\lamax} +o(1).
\]
We combine the previous inequality with \eqref{eq:weak_conv} to obtain
\[
\begin{split}
\left( \frac{1}{\lamax} + o(1) \right)^{(N-2)/2} &\leq
\|\nabla v_{1,n}\|_2^2+\|\nabla v_{2,n}\|_2^2
=\|\nabla u_{1,n}\|_2^2+\|\nabla u_{2,n}\|_2^2  - (\|\nabla \bar u_1 \|_2^2+\|\nabla \bar u_2\|_2^2)+o(1)\\
&\le (\rho_1+\rho_2)\alpha - \lambda_1(\Omega)(\rho_1+\rho_2) +o(1),
\end{split}
\]
as $n\to+\infty$, which contradicts the assumption.
\end{proof}

\subsection{Existence of ground states}\label{sec:abstractexist}

This section is devoted to prove the following result.
\begin{theorem}\label{thm:existence_bad_cond}
Let $\rho_1,\rho_2\ge0$ be such that
\begin{equation}\label{eq:mainassL}
\dd(\rho_1,\rho_2) \cdot (\rho_1+\rho_2)^{a-1} \le  \frac{(a-1)^{a-1}}{a^a}
\lambda_j(\Omega)^{-(a-1)},
\end{equation}
where $j=1$ if $\beta\ge-\sqrt{\mu_1\mu_2}$, $j=2$ otherwise. Let
$\bar \alpha= \frac{a}{a-1}\lambda_i(\Omega)$.

Then $c_{\bar\alpha}$ is achieved by
$(\bar u_1,\bar u_2)\in \Bcal_{\bar \alpha}\setminus \Ucal_{\bar \alpha}$ such that  $\Ecal(\bar u_1, \bar u_2)= c_{\bar \alpha}$, which implies that  $(\bar u_1, \bar u_2)$ is a local minimum of
$\Ecal|_{\Mcal}$, corresponding to a positive solution of \eqref{eq:system_elliptic} for some
$(\omega_1,\omega_2)\in \R^2$. Equivalently, $(\rho_1,\rho_2)\in A$ as defined in \eqref{eq:defA}.
\end{theorem}
First of all, we state a sufficient condition for the above theorem to hold, in terms of $\hat c_\alpha$.
\begin{lemma}\label{lem:c<hat_c}
Let us assume that $\rho_1,\rho_2>0$ are such that, for some $\alpha_1,\alpha_2$,
\[
\lambda_1(\Omega)\le \alpha_1 < \alpha_2
\qquad \text{ and } \qquad
\hat c_{\alpha_1} < \hat c_{\alpha_2};
\]
furthermore, in the Sobolev critical case $N\ge3$, $p=2^*-1$, let us also assume that
\[
\alpha_2 < \lambda_1(\Omega) +
\frac{\lamax^{-(N-2)/2}}{\rho_1+\rho_2}.
\]
Then $c_{\alpha_2} < \hat c_{\alpha_2}$, and $c_{\alpha_2}$ is achieved by a positive solution of \eqref{eq:system_elliptic}.
\end{lemma}
\begin{proof}
Firstly, $c_{\alpha_2}$ is achieved by some $(\bar u_1,\bar u_2)\in\Bcal_{\alpha_2}$: as we already
observed, this is trivial in the Sobolev subcritical case, while in the critical one it follows by
Proposition \ref{prop:Sobcritconv}. Next we observe that
\[
c_{\alpha_2} = \min\left\{\hat c_\alpha:\lambda_1(\Omega)\le \alpha \le \alpha_2\right\}\le
 \hat c_{\alpha_1} < \hat c_{\alpha_2}.
\]
We deduce that $(\bar u_1,\bar u_2)\in\Bcal_{\alpha_2}\setminus\Ucal_{\alpha_2}$, and the lemma
follows.
\end{proof}
Let us denote by $\lambda_2(\Omega)$ the second eigenvalue of $-\Delta$ in $H^1_0(\Omega)$, and by $\varphi_2$ a corresponding eigenfunction.
\begin{lemma}\label{lem:estimate_c}
We have
\[
(\sqrt{\rho_1}\varphi_1,\sqrt{\rho_2} \varphi_1)\in \Ucal_{\lambda_1(\Omega)},\qquad \left(\sqrt{\rho_1}\frac{\varphi_2^+}{\|\varphi_2^+\|_2},\sqrt{\rho_2} \frac{\varphi_2^-}{\|\varphi_2^-\|_2}\right)\in \Ucal_{\lambda_2(\Omega)}.
\]
In particular
\[
\hat c_{\lambda_j(\Omega)} \leq \frac{\rho_1+\rho_2}{2} \lambda_j(\Omega), \qquad
\text{for }j=
\begin{cases}
1 & \text{if }\beta\geq -\sqrt{\mu_1\mu_2},\\
2 & \text{if }\beta< -\sqrt{\mu_1\mu_2}.
\end{cases}
\]
\end{lemma}
\begin{proof}
The first assertion is direct. Then
\[
\begin{split}
\hat c_{\lambda_1(\Omega)}& \leq \Ecal(\sqrt{\rho_1}\varphi_1,\sqrt{\rho_2}\varphi_1) \\&=
\frac{\rho_1+\rho_2}{2}
\lambda_1(\Omega)-\frac{\mu_1 \rho_1^{p+1}+2\beta (\rho_1\rho_2)^\frac{p+1}{2} +\mu_2 \rho_2^{p+1}}{p+1}\int_\Omega \varphi_1^4 \leq \frac{\rho_1+\rho_2}{2}
\lambda_1(\Omega),
\end{split}
\]
since $\beta\ge -\sqrt{\mu_1\mu_2}$ implies that $\mu_1 \rho_1^{p+1}+2\beta (\rho_1\rho_2)^\frac{p+1}{2} +\mu_2 \rho_2^{p+1} \geq 0$ for every $\rho_1,\rho_2>0$.

On the other hand,
\begin{align*}
\hat c_{\lambda_2(\Omega)} &\leq \Ecal\left(\sqrt{\rho_1}\frac{\varphi_2^+}{\|\varphi_2^+\|_{L^2(\Omega)}},\sqrt{\rho_2}\frac{\varphi_2^-}{\|\varphi_2^-\|_{L^2(\Omega)}}\right) \\
 	& \leq \frac{\rho_1+\rho_2}{2} \lambda_2(\Omega)
-\frac{\mu_1\rho_1^{p+1}}{(p+1)\|\varphi_2^+\|_{L^2(\Omega)}^{p+1}}\int_\Omega (\varphi_2^+)^{p+1}\,dx
-\frac{\mu_2\rho_2^{p+1}}{(p+1)\|\varphi_2^-\|_{L^2(\Omega)}^{p+1}}\int_\Omega (\varphi_2^-)^{p+1}\,dx\\
&\leq \frac{\rho_1+\rho_2}{2} \lambda_2(\Omega).  \qedhere
\end{align*}
\end{proof}
\begin{proof}[Proof of Theorem \ref{thm:existence_bad_cond}]
In the following let $j=1$ if $\beta\ge -\sqrt{\mu_1\mu_2}$ and  $j=2$ otherwise.
In view of the application of Lemma \ref{lem:c<hat_c}, our aim is to find $\bar\alpha>\lambda_j(\Omega)$ such that
\begin{equation}\label{eq:c<hat_c_app}
\hat c_{\lambda_j(\Omega)} < \hat c_{\bar\alpha}.
\end{equation}
To start with, we look for a sufficient condition implying \eqref{eq:c<hat_c_app}. 
Using Lemmas \ref{lem:hatc_from_below} and \ref{lem:estimate_c}, it is sufficient to find
$\bar\alpha>\lambda_j(\Omega)$ such that
\[
\frac{\rho_1+\rho_2}{2} \lambda_j(\Omega) \le \frac12\left(\bar\alpha(\rho_1+\rho_2) - \dd(\rho_1,\rho_2) (\rho_1+\rho_2)^a\bar\alpha^a \right)
\]
(by Lemma \ref{lem:hatc_from_below}, the right hand side is strictly less than $\hat c_{\bar\alpha}$).
Equivalently,
\begin{equation}\label{eq:passa}
\dd(\rho_1,\rho_2) (\rho_1+\rho_2)^{a-1} \le \frac{\bar \alpha - \lambda_j(\Omega) }{\bar\alpha^a}.
\end{equation}
By a direct computation, recalling that $a>1$, the best possible choice for the right hand side is
\[
\max_{\alpha\ge\lambda_j(\Omega)}\frac{\alpha - \lambda_j(\Omega) }{\alpha^a} =
\frac{(a-1)^{a-1}}{a^a}
\lambda_j(\Omega)^{-(a-1)}, \qquad\text{achieved by }
\bar \alpha= \frac{a}{a-1}\lambda_j(\Omega).
\]
This choice of $\bar \alpha$ is possible, since it makes \eqref{eq:passa} equivalent to
\eqref{eq:mainassL}, the assumption
of the theorem. Furthermore, it is clear that $\bar \alpha>\lambda_j(\Omega)$. Then, in order to apply
Lemma \ref{lem:c<hat_c} and conclude the
proof, we only need to check that, in case $N\ge3$ and $p=2^*-1$, the additional assumption
\begin{equation}\label{eq:passa2}
\bar\alpha < \lambda_1(\Omega) +
\frac{1}{\lamax^{(N-2)/2}(\rho_1+\rho_2)}
\end{equation}
holds true. This is straightforward since, being $a=N/(N-2)$, relation \eqref{eq:passa} provides
\[
\dd (\rho_1+\rho_2)^{2/(N-2)} \le \frac{\bar \alpha - \lambda_j(\Omega) }{\bar\alpha^{N/(N-2)}}
< \frac{\bar \alpha - \lambda_1(\Omega) }{(\bar \alpha - \lambda_1(\Omega))^{N/(N-2)}} = \frac{1}{(\bar \alpha - \lambda_1(\Omega))^{2/(N-2)}},
\]
which is equivalent to \eqref{eq:passa2}.
\end{proof}

\begin{remark}
The solution $(\bar u_1,\bar u_2)$ does \emph{not} coincide with $(\sqrt{\rho_1} \varphi_1,\sqrt{\rho_2} \varphi_1)$, unless $\beta=-\mu_1=-\mu_2$. Indeed, this last pair solves \eqref{eq:system_elliptic} if and only if, for every $i=1,2$,
\begin{multline*}
(\lambda_1(\Omega)+\omega_i)\sqrt{\rho_i}\varphi_1=(\mu_i+\beta)\rho_i^\frac{p}{2}\varphi_1^{p} \iff
\lambda_1(\Omega)+\omega_i=(\mu_i+\beta)\rho_i^\frac{p-1}{2}\varphi_1^{p-1} \\
\iff \lambda_1(\Omega)=-\omega_i,\ \beta=-\mu_1=-\mu_2.
\end{multline*}
\end{remark}

\subsection{The set \texorpdfstring{$A$}{A} is star-shaped} \label{sec:star}

This section is devoted to the proof of the following result.
\begin{proposition}\label{prop_starshaped}
Let $A$ be defined as in \eqref{eq:defA}. Then $A$ is star-shaped with respect to $(0,0)$.
\end{proposition}
We follow a strategy inspired by \cite{ntvDCDS} but, since such paper does not extend directly
to the Sobolev critical case, we provide here a self-contained argument. In this section it is
convenient to make explicit the dependence of some quantities with respect to $\rho_1,\rho_2$: in view
of this, we write $c_\alpha(\rho_1,\rho_2)$, $\hat c_\alpha(\rho_1,\rho_2)$,
$\Bcal_\alpha(\rho_1,\rho_2)$, $\Ucal_\alpha(\rho_1,\rho_2)$.  For shorter notation, we define
\[
F(u_1,u_2):=\int_\Omega \mu_1 |u_1|^{p+1} + 2\beta |u_1|^{(p+1)/2} |u_2|^{(p+1)/2} + \mu_2 |u_2|^{p+1}
\]
and we introduce the optimization problem
\[
M_\alpha(\rho_1,\rho_2):= \sup _{\Ucal_\alpha(\rho_1,\rho_2)} F
\]
(a quantity thoroughly investigated in \cite{ntvDCDS}). Notice that
\[
\hat c_\alpha(\rho_1,\rho_2)=\frac{1}{2}\alpha(\rho_1+\rho_2)-\frac{1}{p+1} M_\alpha(\rho_1,\rho_2),
\]
and that $\hat c_\alpha(\rho_1,\rho_2)$ is achieved at $(u_1,u_2)\in \Ucal_\alpha(\rho_1,\rho_2)$ if, and only if, $M_\alpha(\rho_1,\rho_2)$ is achieved at the same pair.

Fix, if any, $(\rho_1,\rho_2)\in A\setminus\{(0,0)\}$. By definition of $A$, there exist $\alpha>\lambda_1(\Omega)$ and $(\bar u_1,\bar u_2)\in \Bcal_\alpha$, a solution of \eqref{eq:system_elliptic}, such that $\Ecal(\bar u_1,\bar u_2)=c_{\alpha}< \hat c_{\alpha}$, and $\alpha$ satisfies \eqref{eq:compact_intro} in case $N\ge3$, $p=2^*-1$.
Notice that the assumption $c_{\alpha}< \hat c_{\alpha}$ implies $\int_\Omega |\nabla \bar u_1|^2+|\nabla \bar u_2|^2<(\rho_1+\rho_2)\alpha$, so that
\[
\bar \alpha:=\frac{1}{\rho_1+\rho_2}\int_\Omega |\nabla \bar u_1|^2+|\nabla \bar u_2|^2<\alpha.
\]
As a consequence, $(\bar u_1,\bar u_2)\in \Ucal_{\bar \alpha}(\rho_1,\rho_2)$ achieves $\hat c_{\bar \alpha}
= c_\alpha$.
\begin{lemma}\label{lem:starshap1}
If $s>0$ then $(s \bar u_1,s \bar u_2)\in \Ucal_{\bar \alpha}(s^2\rho_1,s^2\rho_2)$ achieves
\[
\hat c_{\bar \alpha}(s^2 \rho_1,s^2\rho_2)=\frac{s^2}2  \bar \alpha (\rho_1+\rho_2) -
\frac{s^{p+1}}{p+1} F(\bar u_1,\bar u_2).
\]
\end{lemma}
\begin{proof}
This follows by noticing that
\[
(u_1,u_2)\in\Ucal_{\bar\alpha}(\rho_1,\rho_2)
\qquad\iff\qquad
(su_1,su_2)\in\Ucal_{\bar\alpha}(s^2\rho_1,s^2\rho_2),
\]
with
\[
F(su_1,su_2) = s^{p+1} F(u_1,u_2).
\]
Then $M_\alpha(s^2\rho_1,s^2\rho_2) = s^{p+1}M_\alpha(\rho_1,\rho_2)$ and the lemma follows.
\end{proof}
\begin{lemma}\label{lem:curvetta}
Let $s\in (0,1)$ and $(v_1,v_2)\in H^1_0(\Omega,\R^2)$ be such that
\[
\int_\Omega \bar u_1v_1=\int_\Omega \bar u_2 v_2=0,\qquad \int_\Omega \nabla \bar u_1\cdot \nabla v_1+\nabla \bar u_2 \cdot \nabla v_2<0.
\]
Let, for $|t|$ small,
\[
(U_1(t),U_2(t)):=\left( s\sqrt{\rho_1} \frac{\bar u_1+tv_1}{\|\bar u_1+tv_1\|_2},s\sqrt{\rho_2} \frac{\bar u_2+tv_2}{\|\bar u_2+tv_2\|_2}\right).
\]
Then $(U_1(t),U_2(t))\in \Mcal_{s^2\rho_1,s^2\rho_2}$ for every $t$ and
\[
\left.\frac{d}{dt} \|(U_1(t),U_2(t)) \|^2_{H^1_0(\Omega)} \right|_{t=0}<0,\qquad \left.\frac{d}{dt} \Ecal(U_1(t),U_2(t))  \right|_{t=0}<0.
\]
\end{lemma}
\begin{proof}
By direct inspection we have that $(U_1(t),U_2(t))\in \Mcal_{s^2\rho_1,s^2\rho_2}$ for every $t$, and that
\[
\left.\frac{d}{dt} (U_1(t),U_2(t)) \right|_{t=0} = (sv_1,sv_2).
\]
Then
\[
\left.\frac{d}{dt} \|(U_1(t),U_2(t)) \|^2_{H^1_0(\Omega)} \right|_{t=0} =
2s^2\int_\Omega \nabla \bar u_1\cdot \nabla v_1+\nabla \bar u_2 \cdot \nabla v_2<0
\]
by assumption. On the other hand, recalling that $(\bar u_1,\bar u_2)$ solves
\eqref{eq:system_elliptic}, we have that
\[
\left.\frac{d}{dt} F(U_1(t),U_2(t))  \right|_{t=0} = s^{p+1}
F'(\bar u_1,\bar u_2)[v_1,v_2] = s^{p+1}
\int_\Omega \nabla \bar u_1\cdot \nabla v_1+\nabla \bar u_2 \cdot \nabla v_2
\]
and
\[
\left.\frac{d}{dt} \Ecal(U_1(t),U_2(t))  \right|_{t=0} =
(s^2- s^{p+1}) \int_\Omega \nabla \bar u_1\cdot \nabla v_1+\nabla \bar u_2 \cdot \nabla v_2<0,
\]
as $0<s<1$ and $p>1$.
\end{proof}
\begin{proof}[End of the proof of Proposition \ref{prop_starshaped}]
With the notation of Lemma \ref{lem:curvetta}, being
$(U_1(0),U_2(0)) = (s\bar u_1,s\bar u_2)\in\Ucal_{\bar\alpha}(s^2\rho_1,s^2\rho_2)$, there
exist positive and small constants $\varepsilon,\tau$ such that
\[
(U_1(\tau),U_2(\tau)) \in\Ucal_{\bar\alpha-\varepsilon}(s^2\rho_1,s^2\rho_2)
\]
and
\[
\hat c_{\bar\alpha-\varepsilon}(s^2\rho_1,s^2\rho_2) \le \Ecal (U_1(\tau),U_2(\tau))
< \Ecal (U_1(0),U_2(0)) = \hat c_{\bar\alpha}(s^2\rho_1,s^2\rho_2)
\]
(the last equality following by Lemma \ref{lem:starshap1}). Then we can apply Lemma
\ref{lem:c<hat_c}, with $\alpha_1 = \bar\alpha - \varepsilon$ and $\alpha_2 = \bar\alpha$,
obtaining that $(s^2\rho_1,s^2\rho_2)\in A$. Since this holds true for any $s\in(0,1)$, the
proposition follows.
\end{proof}

\begin{remark}
In \cite[Theorem 1.1]{ntvDCDS}, we show that, in the Sobolev subcritical case, $M_{\alpha}(\rho_1,\rho_2)$ is achieved and that an associated maximum point $(u_1,u_2)$ satisfies
\[
\begin{cases}
-\Delta u_1+ \omega_1 u_1=\gamma(\mu_1 u_1|u_1|^{p-1}+\beta u_1|u_1|^{(p-3)/2} |u_2|^{(p+1)/2})\\
-\Delta u_2+ \omega_2 u_2=\gamma(\mu_2 u_2|u_2|^{p-1}+\beta u_2 |u_2|^{(p-3)/2} |u_1|^{p+1/2})\\
\int_\Omega u_i^2=\rho_i, \quad i=1,2,\qquad (u_1,u_2) \in H^1_0(\Omega;\R^2).
\end{cases}
\]
for a suitable Lagrange multiplier $\gamma>0$. Then, repeating the above arguments, we obtain that
\[
\gamma>1 \text{ for some $\alpha$}  \implies  \hat c_{\alpha-\eps}<\hat c_\alpha,
\]
so that Theorem \ref{thm:existence_bad_cond} applies.
\end{remark}

\subsection{Explicit estimates for \texorpdfstring{$\dd$}{Lambda}} \label{sec:explicit_est}

At this point, the main assumption in Theorem \ref{thm:existence_bad_cond} is written in terms of the function $\Lambda(\rho_1,\rho_2)$ defined in \eqref{eq:defdd}, which we recall here for the reader's convenience
\[
\dd(\rho_1,\rho_2)=\frac{2C_{N,p}}{p+1} \max_{t\in [0,\pi/2]}
\left(\mu_1 \rho_1^{2r}\cos^{2a}t
+\mu_2 \rho_2^{2r}\sin^{2a}t
+2\beta^+ \rho_1^r\rho_2^{r}  \cos^{a}t \sin^{a}t\right),
\]
where
\[
a = \frac{N(p-1)}{4} \in \left(1,\frac{N}{N-2}\right],\qquad r =  \frac{p+1}{4} - \frac{N(p-1)}{8}\in \left[0,\frac{1}{N}\right).
\]
It is clear that $\dd$ is a $r$-homogeneous polynomial of $(\rho_1,\rho_2)$, but its explicit expression can be derived only in few particular cases. The aim of this subsection is to prove Theorem \ref{prop:supercritical} by showing that condition \eqref{eq:assnice} in Theorem \ref{prop:supercritical}, with $R=R(\Omega,N,p)$ defined as
\begin{equation}\label{eq:def_R}
R(\Omega,N,p)=\frac{p+1}{2C_{N,p}}\frac{(a-1)^{a-1}}{a^a} \lambda_j(\Omega)^{-(a-1)},
\end{equation}
implies assumption \eqref{eq:mainassL} in Theorem \ref{thm:existence_bad_cond}.
Here, as usual, $j=1$ if $\beta\geq -\sqrt{\mu_1\mu_2}$ and $j=2$ otherwise (or simply $j=2$ for any $\beta$,
in case one wants to avoid this weak dependence of $R$ on $\beta$).
The advantage of \eqref{eq:assnice} with respect to \eqref{eq:mainassL} is that of being more explicit; furthermore, the two conditions coincide in the case $\beta\leq0$, as proved in Remark \ref{rem:beta<0} below.
\begin{proof}[End of the proof of Theorem \ref{prop:supercritical}]
We proved that $A$ is star-shaped with respect to the origin in Proposition \ref{prop_starshaped}.
We estimate $\dd$ from above noticing that, as $a>1$, we have
\[
\begin{split}
\dd(\rho_1,\rho_2)\le\dd'&(\rho_1,\rho_2):=\frac{2C_{N,p}}{p+1} \max_{t\in [0,\pi/2]}
\left(\mu_1 \rho_1^{2r}\cos^{2}t
+\mu_2 \rho_2^{2r}\sin^{2}t
+2\beta^+ \rho_1^r\rho_2^{r}  \cos t \sin t\right)\\
&=\frac{C_{N,p}}{p+1} \max_{t\in [0,\pi/2]}
\left[\left(\mu_1 \rho_1^{2r}+\mu_2 \rho_2^{2r}\right) +
\left(\mu_1 \rho_1^{2r}-\mu_2 \rho_2^{2r}\right)\cos 2t
+2\beta^+ \rho_1^r\rho_2^{r}  \sin 2t\right]\\
&=\frac{C_{N,p}}{p+1} \max_{x^2+y^2=1}
\left[\left(\mu_1 \rho_1^{2r}+\mu_2 \rho_2^{2r}\right) +
\left(\mu_1 \rho_1^{2r}-\mu_2 \rho_2^{2r}\right)x
+2\beta^+ \rho_1^r\rho_2^{r}  y\right].
\end{split}
\]
Next, explicit computations show that
\[
\begin{split}
\dd'(\rho_1,\rho_2) &= \frac{C_{N,p}}{p+1}\left(\mu_1 \rho_1^{2r}+\mu_2 \rho_2^{2r} +
\sqrt{(\mu_1 \rho_1^{2r}-\mu_2 \rho_2^{2r})^2 + 4(\beta^+ \rho_1^r\rho_2^{r})^2 } \right)\\
&\le  \frac{C_{N,p}}{p+1}\left(\mu_1 \rho_1^{2r}+\mu_2 \rho_2^{2r} +
|\mu_1 \rho_1^{2r}-\mu_2 \rho_2^{2r}| + 2\beta^+ \rho_1^r\rho_2^{r}  \right)\\
&= \frac{2C_{N,p}}{p+1}\left[\max\{\mu_1 \rho_1^{2r},\mu_2 \rho_2^{2r}\} +
\beta^+ \rho_1^r\rho_2^{r}  \right].
\end{split}
\]
Therefore assumption \eqref{eq:assnice}, with $R$ as in \eqref{eq:def_R}, implies \eqref{eq:mainassL}, so that we can apply Theorem \ref{thm:existence_bad_cond} to conclude.
\end{proof}
\begin{remark}\label{rem:beta<0}
Relations \eqref{eq:assnice} and \eqref{eq:mainassL} coincide for $\beta\leq 0$. Indeed, in such case, the maximum in
the definition of $\Lambda'$ is achieved when either $t=0$ or $t=\pi/2$, and for such values
the estimate is an equality:
\[
\dd(\rho_1,\rho_2)=\dd'(\rho_1,\rho_2) = \frac{2C_{N,p}}{p+1} \max\{\mu_1 \rho_1^{2r},\mu_2 \rho_2^{2r}\} \qquad \text{ for } \beta\leq 0.
\]
\end{remark}

%%%%%

\section{Orbital stability of the set of ground states}\label{sec:stab}

This section is devoted to the proof of the stability statements, namely of Theorems \ref{prop:subcritical}-b) and \ref{thm:stab}. Our aim is to prove the stability of the sets $G$ and $G_{\bar \alpha}$ defined in the statements.
Actually, in the case of global minimizers (i.e. $1<p\le 1+4/N$), the stability follows from the conservation of the energy and masses, and from the compactness of any minimizing
sequence, see e.g. \cite[Remark 8.3.9]{Cazenave2003}. The case of local minimizers (i.e. $1+4/N < p \le 2^*-1$), however, requires an adaptation of such arguments,
in particular in the Sobolev critical case.

In order to provide a unified proof for all the cases, we first observe that global minimizers
are also local ones. Recall the definitions of $B_\alpha, \Ucal_\alpha$ in \eqref{eq:Balpha}, and those of $c_\alpha,\hat{c}_\alpha$ in  \eqref{eq:calpha}. For $p\leq 1+4/N$ and $(\rho_1, \rho_2)$ satisfying the 
assumptions of Theorem \ref{prop:subcritical}, by the inequalities \eqref{eq:importantestimate} and \eqref{eq:L^2crit_energyestimate}, we readily infer the existence of $\bar \alpha>\lambda_1(\Omega)$ such that $\{(|u_1|,|u_2|)\in H^1_0(\Omega;\R^2):\ (u_1,u_2)\in G\}\subseteq \Bcal_{\bar \alpha}$ and $\inf_{\Mcal}\Ecal=c_{\bar \alpha}< \hat c_{\bar \alpha}$; in particular, $G=G_{\bar \alpha}$.
For $p>1+4/N$ and $(\rho_1, \rho_2)\in A$, take $\bar\alpha\ge \lambda_1(\Omega)$ such that $c_{\bar\alpha}<\hat c_{\bar\alpha}$, as in the statement of Theorem \ref{thm:stab}, and satisfying moreover \eqref{eq:compact_intro} in the case $p=2^*-1$. Therefore, for the previous choice of $\bar \alpha$, we are reduced in all cases to prove the stability of the set $G_{\bar \alpha}$.

To this aim, we recall that a set $\Gcal\subset H^1_0(\Omega;\C^2)$ is orbitally stable if for every
$\eps>0$ there exists $\delta>0$ such that, whenever $(\psi_1,\psi_2)\in H^1_0(\Omega;\C^2)$ satisfies
$\dist_{H^1_0}((\psi_1,\psi_2),\Gcal)<\delta$, $\dist_{H^1_0}$ denoting the $H^1_0$--distance,
then the solution $(\Psi_{1},\Psi_{2})$ of
\[
\begin{cases}
\icomp\partial_t\Psi_1 + \Delta\Psi_1 + \Psi_1( \mu_1 |\Psi_1|^{p-1} +\beta  |\Psi_1|^{(p-3)/2}|\Psi_2|^{(p+1)/2} )=0\\
\icomp\partial_t\Psi_2 + \Delta\Psi_2 + \Psi_2( \mu_2 |\Psi_2|^{p-1} +\beta  |\Psi_2|^{(p-3)/2}|\Psi_1|^{(p+1)/2} )=0\\
\Psi_i(0,\cdot) = \psi_i(\cdot),\qquad \Psi_i(t,\cdot)\in H^1_0(\Omega;\C^2).
\end{cases}
\]
is such that
\begin{equation}\label{eq:os1}
(\Psi_1(t,\cdot),\Psi_2(t,\cdot))\text{ can be continued to a solution in } 0\leq t <+\infty
\end{equation}
and
\begin{equation}\label{eq:os2}
\sup_{t>0} \dist_{H^1_0}((\Psi_1(t,\cdot),(\Psi_2(t,\cdot)),\Gcal) <\eps,
\end{equation}

As we mentioned, we prove the orbital stability of $G_{\bar\alpha}$ under the
condition that for every $M>0$ there exists $T_0=T_0(M)$ such that if
$\|(\psi_{1},\psi_{2})\|_{H^1_0(\Omega;\C^2)}\le M$ then the above Cauchy problem
admits an unique solution on $[0,T_0)$, and that both $\Qcal$ and $\Ecal$ are preserved along the
solutions. Notice that, under these conditions, the failure of \eqref{eq:os1} implies the failure of
\eqref{eq:os2}. Indeed, if \eqref{eq:os1} does not hold, then since $T_0$ depends on the norm of the
initial data we necessarily have $\|(\Psi_{1}(t,\cdot),\Psi_{2}(t,\cdot)\|_{H^1_0(\Omega;\C^2)} \to +\infty$ as $t$
approaches a finite endpoint of the maximal existence time-interval. Since $G_{\bar\alpha}$ is bounded in $H^1_0(\Omega;\C^2)$,
\eqref{eq:os2} cannot hold.
\medbreak

We start with the following preliminary considerations.
\begin{lemma}\label{lemma:c_c'}
Let $(u_1,u_2)\in G_{\bar \alpha}$. Then there exist $\theta_1,\theta_2\in \R$ such that
$(u_1,u_2)=(e^{\icomp\theta_1} |u_1|,e^{\icomp\theta_2}|u_2|)$. In particular,
\[
\inf \left\{ \Ecal(v_1,v_2): (v_1,v_2) \in H^1_0(\Omega;\C^2), (|v_1|,|v_2|)
\in\Bcal_{\bar\alpha}\right\} = c_{\bar \alpha},
\]
while
\[
\inf \left\{ \Ecal(v_1,v_2): (v_1,v_2) \in H^1_0(\Omega;\C^2), (|v_1|,|v_2|)
\in\Ucal_{\bar\alpha}\right\} =: \tilde c_{\bar \alpha}\le \hat c_{\bar \alpha}.
\]
\end{lemma}
\begin{proof}
Given $(v_1,v_2)\in \Ucal_{\bar\alpha}$, we have clearly that $(v_1,v_2) \in H^1_0(\Omega;\C^2)$ and that $(|v_1|,|v_2|)
\in\Ucal_{\bar\alpha}$, so that $\tilde c_{\bar\alpha}\le \hat c_{\bar\alpha}$.

Now let $(u_1,u_2)\in G_{\bar \alpha}$. By the diamagnetic inequality \cite[Theorem 7.21]{LiebLoss}, we have
$\int_\Omega |\nabla |u_i||^2 \le \int_\Omega |\nabla u_i|^2$, for $i=1,2$, so that $c_{\bar \alpha}\le \Ecal(|u_1|,|u_2|) \le \Ecal (u_1,u_2)=c_{\bar \alpha}$. As a consequence, $\int_\Omega |\nabla |u_i||^2 = \int_\Omega |\nabla u_i|^2$, for $i=1,2$ so that equality holds in  the diamagnetic inequality, whence $u_i$ is a complex multiple of $|u_i|$, that is to say $u_i=e^{\icomp \theta_i} |u_i|$ for some $\theta_i\in \R$,
and the rest of the lemma follows.
%
%the equality holding if and only if $u_i$ is a complex multiple of $|u_i|$ (i.e., if and only if $u_i=e^{\icomp \theta_i} |u_i|$ for some $\theta_i\in \R$).
%and the rest of the lemma follows.
\end{proof}
Our general criterion for stability is the following.
\begin{proposition}\label{prop:stability}
Let $\bar \alpha$ be as above. If $c_{\bar \alpha}<\tilde c_{\bar \alpha}$, then $G_{\bar \alpha}$ is (conditionally) orbitally stable.
\end{proposition}
The proof is presented after the following lemma.
\begin{lemma}\label{lem:stability_compact}
Let $\bar \alpha$ be as above. Let $\{(\psi_{1,n},\psi_{2,n})\}\subset H^1_0(\Omega;\C^2)$ satisfy, as $n\to\infty$,
\begin{equation}\label{eq:stability_compact1}
\int_\Omega |\psi_{i,n}|^2 \to\rho_i \quad\text{for }i=1,2, \qquad \Ecal(\psi_{1,n},\psi_{2,n})\to c_{\bar \alpha}
\end{equation}
and, for every $n$ sufficiently large,
%\begin{equation}
%\int_\Omega |\nabla \psi_{1,n}|^2 + |\nabla \psi_{2,n}|^2 \leq (\rho_1+\rho_2)\bar\alpha.
%\end{equation}
\begin{equation}\label{eq:stability_compact2}
\int_\Omega |\nabla \psi_{1,n}|^2 + |\nabla \psi_{2,n}|^2 \leq (\rho_1+\rho_2)\bar\alpha + \text{o}(1).
\end{equation}
Then there exists $(u_1,u_2)\in G_{\bar \alpha}$ such that, up to a subsequence, $(\psi_{1,n},\psi_{2,n})\to (u_1,u_2)$, strongly in $H^1_0(\Omega;\C^2)$.
\end{lemma}
\begin{proof}
By \eqref{eq:stability_compact2} there exists $(\bar\psi_1,\bar\psi_2)\in H^1_0(\Omega;\C^2)$ such that, up to a subsequence, $\psi_{i,n}\rightharpoonup u_i$ weakly in $H^1_0(\Omega;\C)$ and $\psi_{i,n}\to u_i$ in $L^2(\Omega;\C)$ for $i=1,2$, as $n\to+\infty$. Then \eqref{eq:stability_compact1}, \eqref{eq:stability_compact2} and Lemma \ref{lemma:c_c'} provide, for $i=1,2$,
\[
\int_\Omega |u_i|^2 =\rho_i, \qquad
\int_\Omega |\nabla u_1|^2 + |\nabla u_2|^2 \leq (\rho_1+\rho_2)\bar\alpha, \qquad
\Ecal(u_{1},u_{2})\ge c_{\bar \alpha}.
\]

Now, in case $p<2^*-1$, we have that $(\psi_{1,n},\psi_{2,n})\to (u_1,u_2)$ also
in $L^{p+1}(\Omega;\C^2)$. Then
\[
c_{\bar \alpha}\le \Ecal(u_{1},u_{2}) \le \liminf_{n\to+\infty} \Ecal(\psi_{1,n},\psi_{2,n}) = c_{\bar \alpha},
\]
and the strong $H^1_0$ convergence follows, together with the fact that $(u_1,u_2)\in G_{\bar \alpha}$.

On the other hand, in case $p=2^*-1$, the result follows by Proposition \ref{prop:Sobcritconv}:
actually, such proposition is stated for real valued functions, but after Lemma \ref{lemma:c_c'}
it is straightforward to check that its proof holds also for complex valued ones.
\end{proof}

\begin{proof}[Proof of Proposition \ref{prop:stability}]
Suppose by contradiction that $\{(\psi_{1,n},\psi_{2,n})\}\subset H^1_0(\Omega;\C^2)$, $(u_{1,n},u_{2,n})\in G_{\bar \alpha}$ and $\bar\eps>0$ are such that
\begin{equation}\label{eq:psi_to_u}
\lim_{n\to\infty} \|(\psi_{1,n},\psi_{2,n})-(u_{1,n},u_{2,n})\|_{H^1_0(\Omega;\C^2)}=0
\end{equation}
and
\begin{equation}
\sup_{t>0} \dist_{H^1_0}((\Psi_{1,n}(t,\cdot),\Psi_{2,n}(t,\cdot)),G_{\bar \alpha}) \geq 2\bar\eps,
\end{equation}
where $(\Psi_{1,n},\Psi_{2,n})$ is the solution of \eqref{eq:system_schro} with initial condition $(\psi_{1,n},\psi_{2,n})$. Then there exists $\{t_n\}$ such that, letting $\phi_{i,n}(x):=\Psi_{i,n}(t_n,x)$, $i=1,2$,
\begin{equation}\label{eq:stability_phi}
\dist_{H^1_0}((\phi_{1,n},\phi_{2,n}),G_{\bar \alpha}) \geq\bar\eps.
\end{equation}
Let us prove that $\{(\phi_{1,n},\phi_{2,n})\}$ satisfies \eqref{eq:stability_compact1} and \eqref{eq:stability_compact2}. Then Lemma \ref{lem:stability_compact} provides a contradicton to \eqref{eq:stability_phi}, thus concluding the proof.

By Lemma \ref{lem:stability_compact}, $G_{\bar \alpha}$ is compact. Therefore, \eqref{eq:psi_to_u} implies the existence of $(u_1,u_2)\in G_{\bar \alpha}$ such that, up to a subsequence,
\begin{equation}\label{eq:strongconvergence_stability}
 (\psi_{1,n},\psi_{2,n})\to (u_{1},u_{2}) \qquad \text{ in } H^1_0(\Omega;\C^2).
\end{equation}
This, combined with the continuity of Sobolev embeddings, implies that $(\psi_{1,n},\psi_{2,n})$ satisfies \eqref{eq:stability_compact1}. Then the conservation of the mass and of the energy imply that
\[
\int_\Omega |\phi_{i,n}|^2=\int_\Omega |\psi_{i,n}|^2\to \rho_i \text{ for } i=1,2,\quad \text{ and } \quad \Ecal(\phi_{1,n},\phi_{2,n})=\Ecal(\psi_{1,n},\psi_{2,n})\to c_{\bar \alpha},
\]
as $n\to+\infty$,
so that $(\phi_{1,n},\phi_{2,n})$ also satisfies \eqref{eq:stability_compact1}.

To conclude the proof, let us check that, at least for a subsequence, $(\phi_{1,n},\phi_{2,n})$ satisfies \eqref{eq:stability_compact2} , that is, we claim that, for $n$ sufficiently large,
\begin{equation}\label{eq:stability_contr2}
\int_\Omega |\nabla \phi_{1,n}|^2 + |\nabla \phi_{2,n}|^2 \leq(\rho_1+\rho_2) \bar \alpha + \text{o}(1)
\end{equation}
By contradiction, assume there exists $\bar n\in \N$ and $\bar \eps>0$ such that
\[
\int_\Omega |\nabla \phi_{1,n}|^2 + |\nabla \phi_{2,n}|^2 \geq (\rho_1+\rho_2)\bar \alpha +\bar \eps.
\]
Since
\begin{multline*}
\int_\Omega |\nabla \Psi_{1,n}(0,\cdot)|^2 + |\nabla \Psi_{2,n}(0,\cdot)|^2 =\int_\Omega |\nabla \psi_{1,n}|^2 + |\nabla \psi_{2,n}|^2 \\
			\leq \int_\Omega |\nabla u_{1,n}|^2 + |\nabla u_{2,n}|^2 + \text{o}(1) \leq (\rho_1+\rho_2)\bar\alpha + \text{o}(1)
\end{multline*}
for $n$ large, then there exists $\bar t_n\in (0,t_n)$ such that $(\Psi_{1,n}(\bar t_n,\cdot), \Psi_{2,n}(\bar t_n,\cdot))$ satisfies \eqref{eq:stability_compact1} and
\[
\int_\Omega |\nabla \Psi_{1,n}(\bar t_n,\cdot)|^2 + |\nabla \Psi_{2,n}(\bar t_n,\cdot)|^2 =  (\rho_1+\rho_2)\bar \alpha + \text{o}(1)
\]
and in particular \eqref{eq:stability_compact2}. By Lemma \ref{lem:stability_compact} there exists $(\bar u_1,\bar u_2)\in G_{\bar \alpha}$ such that
\[
\int_\Omega |\nabla \bar u_{1}|^2 + |\nabla \bar u_2|^2=(\rho_1+\rho_2)\bar \alpha,
\]
which contradicts the assumption $c_{\bar \alpha}<\tilde c_{\bar \alpha}$.

\end{proof}
We proved Proposition \ref{prop:stability} assuming that $c_{\bar \alpha}<\tilde c_{\bar \alpha}$. We now check that, since $c_{\bar \alpha}<\hat c_{\bar \alpha}$, this assumption is satisfied.
\begin{lemma}\label{lem:c_tilde}
Let $\bar \alpha$ be as above. Then $c_{\bar \alpha}<\tilde c_{\bar \alpha}$.
\end{lemma}
\begin{proof}
If by contradiction $\tilde c=c$, then there exists $\eps_n\to 0$ and  $(v_{1,n},v_{2,n})\in H^1_0(\Omega;\C^2)$ such that
\[
\|(v_{1,n},v_{2,n})\|^2_{H^1_0(\Omega;\C^2)}=\bar\alpha(\rho_1+\rho_2),
\qquad
\int_\Omega |v_{i,n}|^2=\rho_i,
\qquad
c_{\bar \alpha}\leq \Ecal(v_{1,n},v_{2,n})\leq c_{\bar \alpha}+\eps_n,
\]
for every $n$, $i=1,2$. Letting $u_{i,n}:=|v_{i,n}|$, $i=1,2$, the diamagnetic inequality implies
\begin{equation}\label{eq:c_tilde1}
\|(u_{1,n},u_{2,n})\|^2_{H^1_0(\Omega;\R^2)} \leq
\|(v_{1,n},v_{2,n})\|^2_{H^1_0(\Omega;\C^2)}=(\rho_1+\rho_2)\bar\alpha,
\end{equation}
so that $(u_{1,n},u_{2,n})$ is an admissible couple for the minimization problem $c$ and then
\begin{equation}\label{eq:c_tilde2}
c_{\bar \alpha}\leq \Ecal(u_{1,n},u_{2,n})\leq \Ecal(v_{1,n},v_{2,n})\leq c_{\bar \alpha}+\eps_n.
\end{equation}
In particular,
\begin{equation}\label{eq:c_tilde3}
\frac{1}{2}\left(\|(v_{1,n},v_{2,n})\|^2_{H^1_0(\Omega;\C^2)}-\|(u_{1,n},u_{2,n})\|^2_{H^1_0(\Omega;\R^2)}\right)= \Ecal(v_{1,n},v_{2,n})- \Ecal(u_{1,n},u_{2,n})\leq \eps_n.
\end{equation}
Then Lemma \ref{lem:stability_compact} applies to both sequences, yielding both
$(v_{1,n},v_{2,n})\to (v_{1,\infty},v_{2,\infty})$ and
$(u_{1,n},u_{2,n})\to (u_{1,\infty},u_{2,\infty})$, strongly in $H^1_0$. Passing to the limit
in  \eqref{eq:c_tilde2} and \eqref{eq:c_tilde3}, we infer
\[
\Ecal(u_{1,\infty},u_{2,\infty})= c_{\bar \alpha}
\qquad\text{and}\qquad
\|(u_{1,\infty},u_{2,\infty})\|^2_{H^1_0(\Omega;\R^2)}=
\|(v_{1,\infty},v_{2,\infty})\|^2_{H^1_0(\Omega;\C^2)}=\bar \alpha(\rho_1+\rho_2).
\]
Then $(u_{1,\infty},u_{2,\infty})\in \Ucal_{\bar\alpha}$, contradicting the fact that $c_{\bar \alpha}<\hat c_{\bar \alpha}$.
\end{proof}
\begin{proof}[End of the proof of Theorems \ref{prop:compact_intro}-b) and \ref{thm:stab}] Recalling the first paragraph of this section, we have to prove that the set $G_{\bar \alpha}$ is (conditionally) orbitally stable. This is a direct consequence of Proposition \ref{prop:stability} together with Lemma \ref{lem:c_tilde}.
\end{proof}

%%%%%

\section{Asymptotic study as \texorpdfstring{$\beta\to -\infty$}{β→−∞}}\label{sec:segregation}

In this section we prove Theorem \ref{thm_beta-infty}. Let $\mu_1,\mu_2>0$, and take $\rho_1,\rho_2>0$ satisfying
\begin{equation}\label{eq:uniform_beta_proof}
\begin{cases}
\rho_1,\rho_2>0 & \quad \text{ if }1<p<1+4/N, \\ \smallbreak
0<\mu_1\rho_1^\frac{2}{N},\mu_2\rho_2^\frac{2}{N}<\frac{N+2}{NC_{N}}  &\quad \text{ if } p=1+\frac{4}{N}\\ \smallbreak
\max\{ \mu_1 \rho_1^{2r},\mu_2 \rho_2^{2r}\}  \cdot (\rho_1+\rho_2)^{a-1} \le \frac{(a-1)^{a-1}}{a^a}
\lambda_2(\Omega)^{-(a-1)} &\quad  \text{ if }1+\frac{4}{N}<p\leq 2^*-1
\end{cases}
\end{equation}
Observe that all these conditions are independent from $\beta$.  Combining Theorems \ref{prop:subcritical}, \ref{prop:supercritical} and \ref{thm:existence_bad_cond} (see also Remarks \ref{rem:uniform_beta1} and \ref{rem:uniform_beta2}) with the definition of $R$ in \eqref{eq:def_R}, we deduce that, given $\beta<0$, there exist positive functions $u_{1,\beta}, u_{2,\beta}$ and $\omega_{1,\beta},\omega_{2,\beta}\in \R$  such that 
\begin{equation}\label{eq:system_elliptic_beta}
\begin{cases}
-\Delta u_{1,\beta}+ \omega_{1,\beta} u_{1,\beta}=\mu_1 u_{1,\beta}^{p}+\beta u_{1,\beta}^{(p-1)/2} u_{2,\beta}^{(p+1)/2}\\
-\Delta u_{2,\beta}+ \omega_{2,\beta} u_{2,\beta}=\mu_2 u_{2,\beta}^{p}+\beta u_{2,\beta}^{(p-1)/2} u_{1,\beta}^{(p+1)/2}\\
\int_\Omega u_i^2=\rho_i, \quad i=1,2,\\
(u_1,u_2) \in H^1_0(\Omega;\R^2)
\end{cases}
\end{equation}
while
\begin{align}
&\Ecal(u_{1,\beta},u_{2,\beta})=\inf_\Mcal \Ecal  & \text{ if } 1<p\leq 1+\frac{4}{N}\label{eq_uniformalphabar}\\
&\Ecal(u_{1,\beta},u_{2,\beta})=\inf_{\Bcal_{\bar \alpha}} \Ecal,\quad  (u_{1,\beta},u_{2,\beta})\in \Bcal_{\bar \alpha}\setminus \Ucal_{\bar \alpha} &  \text{ if } 1+\frac{4}{N}< p\leq 2^*-1 \label{eq_uniformalphabar2}
\end{align}
where $\bar \alpha:=\frac{a}{a-1}\lambda_2(\Omega)$ in \eqref{eq_uniformalphabar2}.

\begin{lemma}\label{lemma_betainfty_aux}
Under the previous assumptions, there exists a constant $C>0$, independent of $\beta$, such that
\[
\|u_{i,\beta}\|_{H^1_0(\Omega)} + \|u_{i,\beta}\|_{L^\infty(\Omega)} + |\omega_{i,\beta}| \leq C \qquad \text{ for every $\beta<0$, $i=1,2$.}
\]
\end{lemma}
\begin{proof}

 Take  $(\xi_1,\xi_2)\in \Mcal$ (if $p\leq 1+\frac{1}{N}$) or $(\xi_1,\xi_2)\in \Bcal_{\bar \alpha}$ (if $p>1+\frac{4}{N}$), with $\xi_1\cdot \xi_2\equiv 0$ in either case. From \eqref{eq_uniformalphabar}--\eqref{eq_uniformalphabar2} we have
\begin{align}\label{eq:energybound}
\Ecal(u_{1,\beta},u_{2,\beta})\leq \Ecal(\xi_1,\xi_2)=\frac{1}{2}\int_\Omega (|\nabla \xi_{1}|^2+|\nabla \xi_{2}|^2)-\frac{1}{p+1}\int_\Omega (\mu_1|\xi_1|^{p+1}+\mu_2 |\xi_2|^{p+1})=:C_1
\end{align}
where $C_1$ is independent of $\beta<0$.

 From the first statement in \eqref{eq_uniformalphabar2} we deduce that $\{(u_{1,\beta},u_{2,\beta})\}_{\beta<0}$ is uniformly bounded in $H^1_0(\Omega)$ for $p> 1+\frac{4}{N}$. In case $1<p<1+\frac{4}{N}$, the $H^1_0$-boundedness follows combining  \eqref{eq:energybound} with the estimate  \[
\Ecal(u_{1,\beta},u_{2,\beta}) \geq
\|\nabla u_{1,\beta}\|_2^{2a} \left( \frac{1}{2} \|\nabla u_{1,\beta}\|_{2}^{2-2a} -\frac{C_{N,p}}{p+1}{\mu_1\rho_1^{2r}}\right) +
\|\nabla u_{2,\beta}\|_2^{2a} \left( \frac{1}{2} \|\nabla u_{2,\beta}\|_{2}^{2-2a} -\frac{C_{N,p}}{p+1}{\mu_2\rho_2^{2r}}\right)
\]
($a:=N(p-1)/4<1$), which corresponds to \eqref{eq:importantestimate} for $\beta<0$, while for $p=1+\frac{4}{N}$ it follows from  \eqref{eq:energybound} and
\[
\Ecal(u_{1,\beta},u_{2,\beta}) \geq  \frac{1}{2}\left(1-\frac{NC_N \mu_1\rho_1^{2/N}}{N+2}\right)\|\nabla u_{1,\beta}\|_2^2 +  \frac{1}{2}\left(1-\frac{NC_N \mu_2\rho_2^{2/N}}{N+2}\right)\|\nabla u_{2,\beta}\|_2^2.
\]
(see \eqref{eq:L^2crit_energyestimate} with $\beta<0$).

By the Sobolev embedding $H^1_0(\Omega)\hookrightarrow L^{p+1}(\Omega)$, we have that $\{(u_{1,\beta},u_{2,\beta})\}_{\beta<0}$ is uniformly bounded in the $L^{p+1}$-norm. In particular,
\begin{align*}
0\leq \frac{2(-\beta)}{p+1}\int_\Omega (u_{1,\beta} u_{2,\beta})^{(p+1)/2} \leq \Ecal(u_{1,\beta},u_{2,\beta}) + \frac{1}{p+1} \int_\Omega \mu_1 u_{1,\beta}^{p+1}+\mu_2 u_{2,\beta}^{p+1}\leq C_2.
\end{align*}
By testing the first equation in \eqref{eq:system_elliptic_beta} by $u_{1,\beta}$ and the second one by $u_{2,\beta}$, and usign the previous estimates, we have, for $i,j\in \{1,2\}$, $i\neq j$,
\begin{multline*}
\rho_i|\omega_{i,\beta}|=\left|\int_\Omega \mu_i u_{i,\beta}^{p+1}+\beta (u_{1,\beta}u_{2,\beta})^{(p+1)/2}- |\nabla u_{i,\beta}|^2\right| \\
\leq \int_\Omega \mu_i u_{i,\beta}^{p+1}+|\beta| (u_{1,\beta}u_{2,\beta})^{(p+1)/2}+ |\nabla u_{i,\beta}|^2 \le C_3.
\end{multline*}
Now we can use a  Brezis-Kato-Moser type argument exactly as in  \cite[pp. 1264--1265]{MR2928850}, obtaining   uniform $L^\infty$--bounds for $\{(u_{1,\beta},u_{2,\beta})\}_{\beta<0}$.
\end{proof}

\begin{proof}[Proof of Theorem \ref{thm_beta-infty}]
By Lemma \ref{lemma_betainfty_aux}, $\{(u_{1,\beta},v_{2,\beta})\}_{\beta<0}$ satisfies the assumptions of  \cite[Theorems 1.3 and 1.5]{STTZ}. Therefore this sequence is uniformly bounded in $C^{0,\alpha}(\overline \Omega)$ for every $0<\alpha<1$, and there exist $(u_1,u_2)\in C^{0,1}(\overline \Omega)$ with $u_1,u_2\geq 0$ in $\Omega$,  and $(\omega_1,\omega_2)\in \R^2$ such that, up to subsequences, as $\beta\to -\infty$ we have
\[
u_{i,\beta}\to u_i \text{ in } C^{0,\alpha}(\overline \Omega)\cap H^1_0(\Omega), \qquad \omega_i \to \omega_i.
\]
(see also \cite{MR2599456, SZ15}).  By \cite[Theorem 1.2]{DWZ} (which is stated for $p=3$, but holds also for a general $p$ without any extra efford), we have
\[
-\Delta (u_1-u_2) + \omega_1 u_1-\omega_2 u_2 \geq \mu_1 u_1^{p} - \mu_2 u_2^p\quad \text{ and }\quad -\Delta (u_2-u_1) + \omega_2 u_2-\omega u_1 \geq \mu_2 u_2^p-\mu_1 u_1^p \quad \text{ in } \Omega.
\]
We can now conclude by taking $w:=u_1-u_2$.\end{proof}

\section*{Acknowledgments} All authors are partially supported by the project ERC Advanced Grant  2013 n. 339958:
``Complex Patterns for Strongly Interacting Dynamical Systems - COMPAT''.
H. Tavares is partially supported by FCT (Portugal) grant UID/MAT/04561/2013. G. Verzini is partially supported  by the PRIN-2015KB9WPT Grant: ``Variational methods, with applications to problems in mathematical physics and geometry''. B. Noris and G. Verzini are  partially supported  by  the INDAM-GNAMPA group.

\medskip
\small
\begin{flushright}
\noindent \verb"benedetta.noris@u-picardie.fr"\\
Laboratoire Ami\'enois de Math\'ematique Fondamentale et Appliqu\'ee\\
Universit\'e de Picardie Jules Verne\\
33 Rue Saint-Leu, 80039 Amiens (France)
\end{flushright}
\small
\begin{flushright}
\noindent \verb"hrtavares@fc.ul.pt"\\
CMAFCIO \& Departamento de Matem\'atica\\
Faculdade de Ci\^encias da Universidade de Lisboa\\
Edif\'icio C6, Piso 1, Campo Grande 1749--016 Lisboa (Portugal)
\end{flushright}
\small
\begin{flushright}
\noindent \verb"gianmaria.verzini@polimi.it"\\
Dipartimento di Matematica, Politecnico di Milano\\
Piazza Leonardo da Vinci 32, 20133 Milano (Italy)
\end{flushright}

\begin{thebibliography}{10}

\bibitem{agrawal2000}
G.~P. Agrawal.
\newblock {\em Nonlinear fiber optics}.
\newblock Springer, 2000.

\bibitem{MR2302730}
A.~Ambrosetti and E.~Colorado.
\newblock Standing waves of some coupled nonlinear {S}chr{\"o}dinger equations.
\newblock {\em J. Lond. Math. Soc. (2)}, 75(1):67--82, 2007.

\bibitem{AmbrosettiProdiBook}
A.~Ambrosetti and G.~Prodi.
\newblock {\em A primer of nonlinear analysis}, volume~34 of {\em Cambridge
  Studies in Advanced Mathematics}.
\newblock Cambridge University Press, Cambridge, 1993.

\bibitem{MR3777573}
T.~Bartsch and L.~Jeanjean.
\newblock Normalized solutions for nonlinear {S}chr\"odinger systems.
\newblock {\em Proc. Roy. Soc. Edinburgh Sect. A}, 148(2):225--242, 2018.

\bibitem{MR3539467}
T.~Bartsch, L.~Jeanjean, and N.~Soave.
\newblock Normalized solutions for a system of coupled cubic {S}chr\"odinger
  equations on {$\mathbb{R}^3$}.
\newblock {\em J. Math. Pures Appl. (9)}, 106(4):583--614, 2016.

\bibitem{2017arXiv170302832B}
T.~{Bartsch} and N.~{Soave}.
\newblock {Multiple normalized solutions for a competing system of
  Schr\"odinger equations}.
\newblock {\em ArXiv e-prints}, (arXiv:1703.02832), 2017.

\bibitem{MR3639521}
T.~Bartsch and N.~Soave.
\newblock A natural constraint approach to normalized solutions of nonlinear
  {S}chr\"odinger equations and systems.
\newblock {\em J. Funct. Anal.}, 272(12):4998--5037, 2017.

\bibitem{MR2252973}
T.~Bartsch and Z.-Q. Wang.
\newblock Note on ground states of nonlinear {S}chr{\"o}dinger systems.
\newblock {\em J. Partial Differential Equations}, 19(3):200--207, 2006.

\bibitem{MR3638314}
J.~Bellazzini, N.~Boussa\"id, L.~Jeanjean, and N.~Visciglia.
\newblock Existence and stability of standing waves for supercritical {NLS}
  with a partial confinement.
\newblock {\em Comm. Math. Phys.}, 353(1):229--251, 2017.

\bibitem{BonheureJeanjeanNoris}
D.~Bonheure, L.~Jeanjean, and B.~Noris.
\newblock Orbital stability of the ground states for the nonlinear
  {S}chr\"odinger equation with harmonic potential.
\newblock {\em In preparation}, 2018.

\bibitem{BL83}
H.~Br{\'e}zis and E.~H. Lieb.
\newblock A relation between pointwise convergence of functions and convergence
  of functionals.
\newblock {\em Proc. Amer. Math. Soc.}, 88(3):486--490, 1983.

\bibitem{MR709644}
H.~Br\'ezis and L.~Nirenberg.
\newblock Positive solutions of nonlinear elliptic equations involving critical
  {S}obolev exponents.
\newblock {\em Comm. Pure Appl. Math.}, 36(4):437--477, 1983.

\bibitem{Cazenave2003}
T.~Cazenave.
\newblock {\em Semilinear {S}chr\"odinger equations}, volume~10 of {\em Courant
  Lecture Notes in Mathematics}.
\newblock New York University Courant Institute of Mathematical Sciences, New
  York, 2003.

\bibitem{MR2090357}
S.-M. Chang, C.-S. Lin, T.-C. Lin, and W.-W. Lin.
\newblock Segregated nodal domains of two-dimensional multispecies
  {B}ose-{E}instein condensates.
\newblock {\em Phys. D}, 196(3-4):341--361, 2004.

\bibitem{MR3660463}
M.~Cirant and G.~Verzini.
\newblock Bifurcation and segregation in quadratic two-populations mean field
  games systems.
\newblock {\em ESAIM Control Optim. Calc. Var.}, 23(3):1145--1177, 2017.

\bibitem{MR1939088}
M.~Conti, S.~Terracini, and G.~Verzini.
\newblock Nehari's problem and competing species systems.
\newblock {\em Ann. Inst. H. Poincar\'e Anal. Non Lin\'eaire}, 19(6):871--888,
  2002.

\bibitem{DWZ}
E.~N. Dancer, K.~Wang, and Z.~Zhang.
\newblock The limit equation for the {G}ross-{P}itaevskii equations and {S}.
  {T}erracini's conjecture.
\newblock {\em J. Funct. Anal.}, 262(3):1087--1131, 2012.

\bibitem{MR2629888}
E.~N. Dancer, J.~Wei, and T.~Weth.
\newblock A priori bounds versus multiple existence of positive solutions for a
  nonlinear {S}chr{\"o}dinger system.
\newblock {\em Ann. Inst. H. Poincar{\'e} Anal. Non Lin{\'e}aire},
  27(3):953--969, 2010.

\bibitem{FibichMerle2001}
G.~Fibich and F.~Merle.
\newblock Self-focusing on bounded domains.
\newblock {\em Phys. D}, 155(1-2):132--158, 2001.

\bibitem{Fukuizumi2012}
R.~Fukuizumi, F.~H. Selem, and H.~Kikuchi.
\newblock Stationary problem related to the nonlinear schr{\"o}dinger equation
  on the unit ball.
\newblock {\em Nonlinearity}, 25(8):2271, 2012.

\bibitem{MR3534090}
T.~Gou and L.~Jeanjean.
\newblock Existence and orbital stability of standing waves for nonlinear
  {S}chr\"odinger systems.
\newblock {\em Nonlinear Anal.}, 144:10--22, 2016.

\bibitem{0951-7715-31-5-2319}
T.~Gou and L.~Jeanjean.
\newblock Multiple positive normalized solutions for nonlinear schr{\"o}dinger
  systems.
\newblock {\em Nonlinearity}, 31(5):2319, 2018.

\bibitem{GrillakisShatahStrauss}
M.~Grillakis, J.~Shatah, and W.~Strauss.
\newblock Stability theory of solitary waves in the presence of symmetry, i.
\newblock {\em Journal of Functional Analysis}, 74(1):160--197, 1987.

\bibitem{MR1430506}
L.~Jeanjean.
\newblock Existence of solutions with prescribed norm for semilinear elliptic
  equations.
\newblock {\em Nonlinear Anal.}, 28(10):1633--1659, 1997.

\bibitem{Kwong1989}
M.~K. Kwong.
\newblock Uniqueness of positive solutions of {$\Delta u-u+u^p=0$} in {${\bf
  R}^n$}.
\newblock {\em Arch. Rational Mech. Anal.}, 105(3):243--266, 1989.

\bibitem{LiebLoss}
E.~H. Lieb and M.~Loss.
\newblock {\em Analysis}, volume~14 of {\em Graduate Studies in Mathematics}.
\newblock American Mathematical Society, Providence, RI, second edition, 2001.

\bibitem{MR2135447}
T.-C. Lin and J.~Wei.
\newblock Ground state of {$N$} coupled nonlinear {S}chr\"odinger equations in
  {$\mathbf{R}^n$}, {$n\leq 3$}.
\newblock {\em Comm. Math. Phys.}, 255(3):629--653, 2005.

\bibitem{MR2358296}
T.-C. Lin and J.~Wei.
\newblock Erratum: ``{G}round state of {$N$} coupled nonlinear {S}chr\"odinger
  equations in {${\bf R}^n$}, {$n\leq3$}'' [{C}omm. {M}ath. {P}hys. {\bf 255}
  (2005), no. 3, 629--653; mr2135447].
\newblock {\em Comm. Math. Phys.}, 277(2):573--576, 2008.

\bibitem{MR2263573}
L.~A. Maia, E.~Montefusco, and B.~Pellacci.
\newblock Positive solutions for a weakly coupled nonlinear {S}chr{\"o}dinger
  system.
\newblock {\em J. Differential Equations}, 229(2):743--767, 2006.

\bibitem{MR2599456}
B.~Noris, H.~Tavares, S.~Terracini, and G.~Verzini.
\newblock Uniform {H}{\"o}lder bounds for nonlinear {S}chr{\"o}dinger systems
  with strong competition.
\newblock {\em Comm. Pure Appl. Math.}, 63(3):267--302, 2010.

\bibitem{MR2928850}
B.~Noris, H.~Tavares, S.~Terracini, and G.~Verzini.
\newblock Convergence of minimax structures and continuation of critical points
  for singularly perturbed systems.
\newblock {\em J. Eur. Math. Soc. (JEMS)}, 14(4):1245--1273, 2012.

\bibitem{ntvAnPDE}
B.~Noris, H.~Tavares, and G.~Verzini.
\newblock Existence and orbital stability of the ground states with prescribed
  mass for the {$L^2$}-critical and supercritical {NLS} on bounded domains.
\newblock {\em Anal. PDE}, 7(8):1807--1838, 2014.

\bibitem{ntvDCDS}
B.~Noris, H.~Tavares, and G.~Verzini.
\newblock Stable solitary waves with prescribed {$L^2$}-mass for the cubic
  {S}chr\"odinger system with trapping potentials.
\newblock {\em Discrete Contin. Dyn. Syst.}, 35(12):6085--6112, 2015.

\bibitem{MR3689156}
D.~Pierotti and G.~Verzini.
\newblock Normalized bound states for the nonlinear {S}chr\"odinger equation in
  bounded domains.
\newblock {\em Calc. Var. Partial Differential Equations}, 56(5):Art. 133, 27,
  2017.

\bibitem{RoseWeinstein88}
H.~A. Rose and M.~I. Weinstein.
\newblock On the bound states of the nonlinear {S}chr\"odinger equation with a
  linear potential.
\newblock {\em Phys. D}, 30(1-2):207--218, 1988.

\bibitem{Sirakov2007}
B.~Sirakov.
\newblock Least energy solitary waves for a system of nonlinear {S}chr\"odinger
  equations in {$\mathbb{R}^n$}.
\newblock {\em Comm. Math. Phys.}, 271(1):199--221, 2007.

\bibitem{STTZ}
N.~Soave, H.~Tavares, S.~Terracini, and A.~Zilio.
\newblock H{\"o}lder bounds and regularity of emerging free boundaries for
  strongly competing {S}chr{\"o}dinger equations with nontrivial grouping.
\newblock {\em Nonlinear Anal.}, 138:388--427, 2016.

\bibitem{SZ15}
N.~Soave and A.~Zilio.
\newblock Uniform bounds for strongly competing systems: the optimal
  {L}ipschitz case.
\newblock {\em Arch. Ration. Mech. Anal.}, 218(2):647--697, 2015.

\bibitem{Struwebook}
M.~Struwe.
\newblock {\em Variational methods}, volume~34 of {\em Ergebnisse der
  Mathematik und ihrer Grenzgebiete. 3. Folge. A Series of Modern Surveys in
  Mathematics [Results in Mathematics and Related Areas. 3rd Series. A Series
  of Modern Surveys in Mathematics]}.
\newblock Springer-Verlag, Berlin, fourth edition, 2008.
\newblock Applications to nonlinear partial differential equations and
  Hamiltonian systems.

\bibitem{PhysRevLett.81.5718}
E.~Timmermans.
\newblock Phase separation of bose-einstein condensates.
\newblock {\em Phys. Rev. Lett.}, 81:5718--5721, Dec 1998.

\bibitem{Weinstein1983}
M.~I. Weinstein.
\newblock Nonlinear {S}chr\"odinger equations and sharp interpolation
  estimates.
\newblock {\em Comm. Math. Phys.}, 87(4):567--576, 1982/83.

\end{thebibliography}
\end{document}